\documentclass[11pt]{article}

\usepackage{amsmath}
\usepackage{amsthm}
\usepackage{amssymb}
\usepackage{latexsym}
\usepackage{textcomp}
\usepackage{mathtools}
\usepackage[T1]{fontenc}
\usepackage[sc]{mathpazo}
\usepackage{enumerate}
\usepackage{titlefoot}
\usepackage{titlesec}
\usepackage[small,it]{caption}
\usepackage{pstricks}
\usepackage{graphicx, pst-plot, pst-node, pst-text, pst-tree, pstricks-add}
\usepackage{color}
\usepackage{calc}
\usepackage{gastex}
\definecolor{lightgray}{rgb}{0.9, 0.9, 0.9}
\definecolor{darkgray}{rgb}{0.7, 0.7, 0.7}
\definecolor{darkblue}{rgb}{0, 0, .4}

\usepackage[bookmarks]{hyperref}
\hypersetup{
    colorlinks=true,
    linkcolor=darkblue,
    anchorcolor=darkblue,
    citecolor=darkblue,
    urlcolor=darkblue,
    pdfpagemode=UseThumbs,
    pdftitle={The enumeration of three pattern classes},
    pdfsubject={Combinatorics},
    pdfauthor={M. Albert, M. Atkinson and R. Brignall},
    pdfkeywords={todo}
}
\parskip\medskipamount
\theoremstyle{plain}
\newtheorem{theorem}{Theorem}[section]

\newtheorem{lemma}[theorem]{Lemma}
\newtheorem{proposition}[theorem]{Proposition}

\theoremstyle{definition}

\theoremstyle{remark}

\newcounter{todocounter}

\makeatletter
\newfont{\footsc}{cmcsc10 at 8truept}
\newfont{\footbf}{cmbx10 at 8truept}
\newfont{\footrm}{cmr10 at 10truept}
\newcommand{\Av}{\operatorname{Av}}

\newcommand{\A}{\mathcal{A}}
\newcommand{\B}{\mathcal{B}}
\newcommand{\C}{\mathcal{C}}
\newcommand{\D}{\mathcal{D}}
\newcommand{\E}{\mathcal{E}}
\newcommand{\F}{\mathcal{F}}
\newcommand{\G}{\mathcal{G}}

\newcommand{\I}{\mathcal{I}}

\renewcommand{\S}{\mathcal{S}}
\newcommand{\T}{\mathcal{T}}

\newcommand{\X}{\mathcal{X}}
\renewcommand{\a}{a}%{\mathbf{a}}
\renewcommand{\b}{b}%{\mathbf{b}}
\renewcommand{\c}{c}%{\mathbf{c}}
\renewcommand{\d}{d}%{\mathbf{d}}

\newcounter{hone}
\newcounter{htwo}
\newcounter{vone}
\newcounter{vtwo}

\def\grid(#1,#2){%
\setcounter{hone}{#1*20}%
\setcounter{vone}{#2*20}%
\psset{xunit=0.005in, yunit=0.005in} \psset{linewidth=0.005in}%
\begin{pspicture}(0,0)(\thehone,\thevone)\drawlines(#1,#2)\end{pspicture}}

\def\drawlines(#1,#2){%
\setcounter{hone}{#1}%
\setcounter{vtwo}{20*\thehone}%
\addtocounter{hone}{1}%
\setcounter{vone}{#2}%
\setcounter{htwo}{20*\thevone}%
\addtocounter{vone}{1}%
\multido{\i=0+20}{\thehone}{\psline[linestyle=solid,linewidth=0.005in](\i,0)(\i,\thehtwo)}%
\multido{\i=0+20}{\thevone}{\psline[linestyle=solid,linewidth=0.005in](0,\i)(\thevtwo,\i)}}

\def\upline(#1,#2){%
\setcounter{hone}{(#1*20)-17}%
\setcounter{vone}{(#2*20)-17}%
\setcounter{htwo}{(#1*20)-3}%
\setcounter{vtwo}{(#2*20)-3}%
\psline[linestyle=solid,linewidth=0.02in](\thehone,\thevone)(\thehtwo,\thevtwo)}

\def\downline(#1,#2){%
\setcounter{hone}{(#1*20)-17}%
\setcounter{vtwo}{(#2*20)-17}%
\setcounter{htwo}{(#1*20)-3}%
\setcounter{vone}{(#2*20)-3}%
\psline[linestyle=solid,linewidth=0.02in](\thehone,\thevone)(\thehtwo,\thevtwo)}

\def\cellclass(#1,#2)#3{%
\setcounter{hone}{(#1*20)-10}%
\setcounter{vone}{(#2*20)-10}%
\rput[c](\thehone,\thevone){#3}}

\title{The enumeration of three pattern classes}
\author{
M.~H. Albert\thanks{{\tt malbert@cs.otago.ac.nz}},
M.~D. Atkinson\thanks{{\tt mike@cs.otago.ac.nz}}\\
Department of Computer Science\\
University of Otago,  New Zealand\\
and\\
Robert Brignall\thanks{{\tt r.brignall@open.ac.uk}}\\
Department of Mathematics\\
The Open University, UK}

\begin{document}
\maketitle
\begin{abstract}The structure of three pattern classes $\Av(2143, 4321)$, $\Av(2143, 4312)$ and $\Av(1324, 4312)$ is determined using the machinery of monotone grid classes.  This allows the permutations in these classes to be described in terms of simple diagrams and regular languages and, using this, the rational generating functions which enumerate these classes are determined.
\end{abstract}
\section{Introduction}

A \emph{pattern class} (or simply \emph{class}) is a set of permutations closed downwards under a natural partial order. Therefore, every class is defined by a set of forbidden restrictions, the minimal permutations that do not belong to the class.  A central endeavor in the theory of pattern classes is to describe the class defined by a given set of restrictions.  To this end a rich structure theory of classes has been developed.  When new theoretical breakthroughs occur they are exploited to widen our understanding to previously inaccessible pattern classes.  In essence that is the story of this paper.  As we shall see the new theory of  \emph{monotone grid classes} (outlined in the next section) allows us to find the structure of (and enumerate) three pattern classes drawn from a particular suite of pattern classes that, over the years, has succumbed gradually to one new technique after another.

We shall be concerned with classical pattern containment.  In other words, one permutation $\pi$ is  contained as a pattern in  another permutation $\sigma$ if $\sigma$ has a subsequence whose elements are ordered relatively the same as the elements of $\pi$.  When permutations are represented by their diagrams (the plot of pairs $(i,\sigma(i))$ in the plane) then pattern containment is essentially containment of point subsets.  Pattern containment is obviously a partial order on the set of all permutations and the down-sets of this poset are called \emph{pattern classes}.

The extensive literature on pattern classes contains studies of many particular classes either because they arise naturally in some other guise or because their study is part of a systematic classification of various types of classes.  Examples of the former case are the classes associated with various data structures \cite{knuth:the-art-of-comp:1} and the classes corresponding to smooth Schubert varieties \cite{bona:the-permutation:}.  This paper however is concerned more with systematic classification. That is, its primary message is that the tools which have been developed that describe certain well-structured classes can be applied directly to completely characterize, and in particular enumerate, classes which had resisted ad hoc attacks.

Every pattern class can be defined by its \emph{basis} which is the set of minimal permutations not in the class.  A permutation will lie in the class if and only if it does not contain any of the basis permutations as a pattern.  We write $\Av(B)$ for the class with basis $B$.  In most cases, it is impossible to say much about the structure of a class given only its basis.  But when the basis contains short permutations (or many permutations) some progress is often possible.  It would not be unfair to say that the structure is completely understood when all the basis permutations have length at most three \cite{simion:restricted-perm:}, or if there are two basis permutations of lengths three and four \cite{atkinson:restricted-perm:,west:generating-tree:96}.  The cases of pattern classes having two basis permutations of length four lie on the boundary of our present knowledge.

%Much is known, for example, if the basis has only permutations of length 3 or less \cite{simion:restricted-perm:}, or if there are two basis permutations of lengths 3 and 4 \cite{atkinson:restricted-perm:,west:generating-tree:96}.  The case where a pattern class has two basis permutations of length 4 lies on the boundary of our present knowledge.

There are $\binom{24}{2}=276$ pairs of distinct permutations of length four but the natural symmetries of the pattern containment order mean that only 56 of them are essentially different.  Within this set of 56  there are only 38 different \emph{Wilf classes} (classes whose enumerations are different) \cite{bona:the-permutation:,kremer:permutations-wi:,kremer:postscript:-per:,kremer:finite-transiti:,le:wilf-classes-of:}.  Of these 38 Wilf classes, about half have been enumerated (see \cite{wiki:enumerations} which lists 18 enumerations).  The main results of this paper produce enumerations (given by generating functions) and structure results for three previously unenumerated  classes: $\Av(2143, 4321)$, $\Av(2143, 4312)$ and $\Av(1324, 4312)$.

We reiterate that, while it is interesting to continue to extend the systematic classification of pattern classes having few short basis elements, we view the main contribution of this work not so much as being the particular structural results and enumerations obtained, but rather the illustration that they provide of our growing understanding of the structure of pattern classes as a whole, particularly of monotone grid classes as discussed in Section \ref{SEC_Monotone_grid_classes}.

To conclude this section we recall the idea of a simple permutation, how every permutation is formed by inflation from a simple permutation, and some relevant terminology.

A permutation is said to be \emph{simple} if it has no non-trivial interval (a contiguous subsequence consisting of contiguous values).  Every permutation $\tau$ is the inflation of a (unique) simple permutation $\sigma$ in the sense that $\tau$ is obtained from $\sigma$ by replacing its elements by intervals (with appropriate normalization)~\cite{albert:simple-permutat:}.  This process is called \emph{inflating} $\sigma$.  An example should clarify these ideas.  If the simple permutation $3142$ is inflated so that its elements $3,1,4,2$ become intervals isomorphic to $12, 213, 1, 312$ respectively we obtain the permutation
\[3142[12, 213, 1, 312]=782139645.\]

While $\tau$ determines $\sigma$ uniquely the interval decomposition of $\tau$ is unique only if $|\sigma|> 2$. The case $|\sigma|=2$ is somewhat different.  These cases are associated with \emph{sum-decomposable} permutations ($\sigma=12$) and \emph{skew-decomposable} permutations ($\sigma=21$).  A sum-decomposable (respectively skew-decomposable) permutation is one of the form $\alpha\beta$ where every element of the prefix $\alpha$ is less (respectively, greater) than every element of the suffix $\beta$.  The terms sum-indecomposable and skew-indecomposable have the obvious meanings while a permutation which is neither sum-indecomposable nor skew-indecomposable is said to be \emph{strong}-indecomposable.

Simple permutations enter into the study of pattern classes in the following way.  If we wish to understand the permutations of a class $\T$ then we may first find its set of simple permutations.  If we know enough about the structure of the simple permutations we shall be able to find the inflations of them that also lie in $\T$.  In general this type of structural understanding is often accompanied by the means to enumerate the class (roughly speaking, the generating function of $\T$ is related by functional composition to the generating function of the simple permutations in $\T$).  This is exactly the the way we shall exploit simple permutations in this paper.

In the next section we summarize a newer technique for analyzing pattern classes.  Then the next three sections treat the pattern classes $\Av(2143, 4321)$, $\Av(2143, 4312)$ and $\Av(1324, 4312)$ one by one.  Since our treatments have some variation we briefly outline them in turn.

\begin{description}
\item[$\Av(2143, 4321)$.] We prove it is the union of 4 monotone grid classes and enumerate it directly from this.
\item[$\Av(2143, 4312)$.] We first prove it is contained in a certain $2\times 2$ monotone grid class and then realize it as the union of two explicit component grid classes.  Next we enumerate the simple permutations in each of these components.  Then we refine this by classifying the simples into a small number of types and enumerating the numbers of each type.  The types are necessary because they have slightly different allowable inflations.  Knowing the inflations allows us to enumerate the strong-indecomposables.  Finally we incorporate the sum- and skew-decomposables.
\item[$\Av(1324, 4312)$.] This is the most complex enumeration because the monotone grid components of the entire class are not described explicitly.  Instead we describe the simple permutations as members of four monotone grid classes.  By enumerating the simples of each type and determining their possible inflations we can enumerate the strong-indecomposables.  Lastly we incorporate the sum- and skew-decomposables.
\end{description}

The discussion of the three classes proceeds roughly in order of the complexity of the underlying structural description. We hope that by presenting them in this order, and being quite explicit in our discussion of the simpler cases we will gain some indulgence from the reader for omitting some of the less interesting detailed justifications of the structural and enumerative results in the more complex ones.

\section{Monotone grid classes}
\label{SEC_Monotone_grid_classes}
In broad terms our approach to each of the pattern classes analyzed in this paper is in two stages.  In the first stage we prove a structure result that relates the pattern class or its simple permutations to a collection of monotone grid classes (defined below).  In the second stage we use an encoding that represents the permutations of these monotone grid classes by  regular languages over a finite alphabet.  In this section we give the background that underpins both of these stages.

In the study of pattern classes it is common to represent permutations $\pi$ by the plot of the points $(i,\pi(i))$ in the plane.  Thus $\pi=[16,13,18,11,19,10,8,7,12,3,2,14,9,6,5,4,1]$ would be represented by

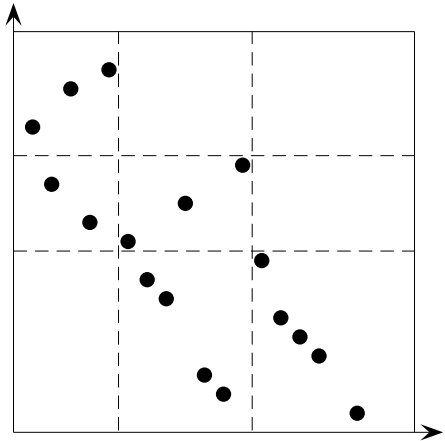
\begin{center}
\psset{xunit=0.01in, yunit=0.01in} \psset{linewidth=0.005in}
\begin{pspicture}(0,0)(225,225)
\psset{arrowsize=4pt 6}
\psline[linestyle=solid,linewidth=0.005in]{->}(0,0)(0,225)
\psline[linestyle=solid,linewidth=0.005in]{->}(0,0)(225,0)
\psline[linestyle=solid,linewidth=0.005in](0,210)(210,210)
\psline[linestyle=solid,linewidth=0.005in](210,0)(210,210)

\pscircle*(10,160){0.04in}
\pscircle*(20,130){0.04in}
\pscircle*(30,180){0.04in}
\pscircle*(40,110){0.04in}
\pscircle*(50,190){0.04in}
\pscircle*(60,100){0.04in}
\pscircle*(70,80){0.04in}
\pscircle*(80,70){0.04in}
\pscircle*(90,120){0.04in}
\pscircle*(100,30){0.04in}
\pscircle*(110,20){0.04in}
\pscircle*(120,140){0.04in}
\pscircle*(130,90){0.04in}
\pscircle*(140,60){0.04in}
%\pscircle*(150,200){0.04in}
\pscircle*(150,50){0.04in}
\pscircle*(160,40){0.04in}
%\pscircle*(180,170){0.04in}
\pscircle*(180,10){0.04in}
%\pscircle*(200,150){0.04in}

\psline[linestyle=dashed,linewidth=0.005in](0,95)(210,95)
\psline[linestyle=dashed,linewidth=0.005in](0,145)(210,145)
\psline[linestyle=dashed,linewidth=0.005in](55,)(55,210)
\psline[linestyle=dashed,linewidth=0.005in](125,)(125,210)

\end{pspicture}
\end{center}

In this diagram the enclosing boundary square has been partitioned into cells, some empty, but all monotonic (increasing or decreasing).  The pattern of empty, increasing and decreasing cells can be represented by a {\em gridding matrix} with 0,1,-1 entries for each type of cell, or by a  cell diagram where the increasing and decreasing cells are displayed as sloping lines.  In this case the matrix  is

\[G=
 \begin{pmatrix*}[r]
  1 & 0 & 0 \\
  -1 & 1 & 0 \\
  0 & -1 & -1
 \end{pmatrix*}
\]

and the cell diagram is

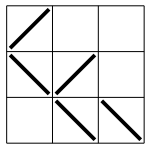
\begin{center}
\psset{xunit=0.006in, yunit=0.006in} \psset{linewidth=0.005in}
\begin{pspicture}(0,0)(120,120)
\psline[linestyle=solid,linewidth=0.005in](0,0)(0,120)
\psline[linestyle=solid,linewidth=0.005in](0,0)(120,0)
\psline[linestyle=solid,linewidth=0.005in](120,0)(120,120)
\psline[linestyle=solid,linewidth=0.005in](0,120)(120,120)
\psline[linestyle=solid,linewidth=0.005in](0,40)(120,40)
\psline[linestyle=solid,linewidth=0.005in](0,80)(120,80)
\psline[linestyle=solid,linewidth=0.005in](40,0)(40,120)
\psline[linestyle=solid,linewidth=0.005in](80,0)(80,120)

\psline[linestyle=solid,linewidth=0.02in](3,83)(37,117)
\psline[linestyle=solid,linewidth=0.02in](3,77)(37,43)
\psline[linestyle=solid,linewidth=0.02in](43,37)(77,3)
\psline[linestyle=solid,linewidth=0.02in](43,43)(77,77)
\psline[linestyle=solid,linewidth=0.02in](83,37)(117,3)

\end{pspicture}
\end{center}

The set of all permutations whose diagram can be represented in this form is the monotone grid class associated with $G$.

The theory of monotone grid classes has been developed by several authors.  The classes were introduced initially by Murphy and Vatter \cite{murphy:profile-classes:} who gave an important necessary and sufficient condition that identifies when a monotone grid class is partially well-ordered.  Their condition can be given in terms of the {\em cell graph} of the defining matrix in which the vertices are the non-empty cells and two cells are adjacent if they lie in the same row or the same column of the matrix with no intervening non-empty cell.  They showed that the associated monotone grid class is partially well-ordered if and only if the cell graph is a forest.  Subsequently Vatter and Waton \cite{vatter:on-partial-well:} simplified the proof and introduced a method of associating words over a finite alphabet with permutations in the monotone grid class which is particularly useful when the cell graph is a forest.  This association can be used to encode permutations as words in a regular language and thereby obtain the generating function of the monotone grid class as a rational function.  It is our main tool in carrying out the enumerations.

Rather than repeat the formal justifications as given in \cite{vatter:on-partial-well:} (and more extensively refined in \cite{albert:geometric-grid-:}) we merely illustrate the general approach with  respect to the above monotone grid class.

The points in any given cell are encoded in the coding word by a common letter.  The order of the code letters is defined by the order in which the points of each cell are `read' and this order is in turn defined by a common order in which the cells in each row and in each column are read.  So, in the example, the reading order can be given by

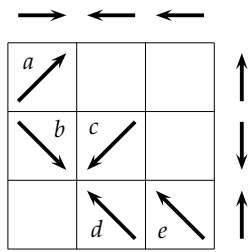
\begin{center}
\psset{xunit=0.018in, yunit=0.018in} \psset{linewidth=0.005in}
\begin{pspicture}(0,0)(80,80)
%\drawlines(4,6)
\drawlines(3,3)
%\psline[linestyle=solid,linewidth=0.02in]{->}(3,3)(17,17)
\psline[linestyle=solid,linewidth=0.02in]{->}(37,37)(23,23)
\psline[linestyle=solid,linewidth=0.02in]{->}(3,43)(17,57)
\psline[linestyle=solid,linewidth=0.02in]{->}(3,37)(17,23)
\psline[linestyle=solid,linewidth=0.02in]{<-}(43,17)(57,3)
\psline[linestyle=solid,linewidth=0.02in]{<-}(23,17)(37,3)

\psline[linestyle=solid,linewidth=0.02in]{->}(3,68)(17,68)
\psline[linestyle=solid,linewidth=0.02in]{<-}(23,68)(37,68)
\psline[linestyle=solid,linewidth=0.02in]{<-}(43,68)(57,68)
\psline[linestyle=solid,linewidth=0.02in]{->}(68,3)(68,17)
\psline[linestyle=solid,linewidth=0.02in]{<-}(68,23)(68,37)
\psline[linestyle=solid,linewidth=0.02in]{->}(68,43)(68,57)

\rput(15,35){\footnotesize $b$}
\rput(25,35){\footnotesize $c$}
\rput(6,54){\footnotesize $a$}
\rput(45,5){\footnotesize $e$}
\rput(26,6){\footnotesize $d$}
\end{pspicture}
\end{center}
%%%
The main technical difficulty is that permutations may have more than one valid decomposition into cells compatible with the gridding matrix. This is usually dealt with by an implicit choice of a particular grid decomposition; in most cases this corresponds to a local constraint on certain combinations of letters.  For example in the diagram of $\pi$ above there is another gridding in which the top horizontal dashed line is a little higher placing the first point into a lower cell. We might have chosen to prefer the given decomposition because the associated word is lexicographically earlier. That this, or a similar choice of  grid decomposition, still leads to a regular language of representative words is one of the main results of \cite{albert:geometric-grid-:}.

Another coding consideration arises when we are  encoding merely the simple permutations in the monotone grid class.  Simplicity can be violated in several ways.  One way is that the maximum or minimum point occurs at one end of the permutation (as happens in the example) and this is handled by forbidding certain beginnings and endings to a code word (in the example, an initial code letter $e$ would not be permitted).  Another is that a pair of points in a cell might form an interval of length 2.  This manifests itself by a repeated code letter and so repeated code letters have to be forbidden.  Still another, often the trickiest to handle, is that an interval can be formed by points in multiple cells.  Again it is established in \cite{albert:geometric-grid-:} that, when the cell graph is a forest (indeed, in a slightly wider context), an encoding for the simple permutations of the corresponding class exists where the codewords form a regular language and so can be defined by a finite automaton; all the encodings that we shall use are defined in this way.

We use the theory of monotone grid classes in another way too through a result of Huczynska and Vatter \cite{huczynska:grid-classes-an:}.  They proved that if a pattern class does not contain permutations which are a sum of 21's or a skew sum of 12's  then it is necessarily contained in some monotone grid class.  In the pattern classes we are considering we have two basis elements of length 4 and we focus on those pattern classes whose basis elements forbid sums of 21's and skew sums of 12's.  This is the reason why our paper is about the 3 pattern classes in question: they are the remaining unenumerated classes with two basis elements of length 4 that satisfy the Huczynska-Vatter criterion.

Finally, we give a name to an elementary monotone grid class.  The class defined by the matrix $(1\ -1)$ consists of permutations $\pi$ which can be written as the concatenation of an increasing sequence and a decreasing sequence, so is denoted by $\wedge$ and is called a \emph{wedge} class.  The basis of the class is $\{213, 312\}$.  The symmetries of $\wedge$ (denoted by $\vee$, $<$ and $>$) are also called wedge classes.

\section{Av(2143, 4321)}

\subsection{The structure of Av(2143, 4321)}

\begin{theorem}\label{maintheorem}
$\Av(4321,2143)$ is the union of four monotone grid classes which are displayed in the following diagram.
\end{theorem}

%\begin{center}
%\includegraphics[width=3in]{Theorem}
%%\caption{The monotone grid classes $A$, $B$, $C$ and $D$ of Theorem \ref{maintheorem}}
%\label{result}
%\end{center}
%Diagram of classes A, B, C, D
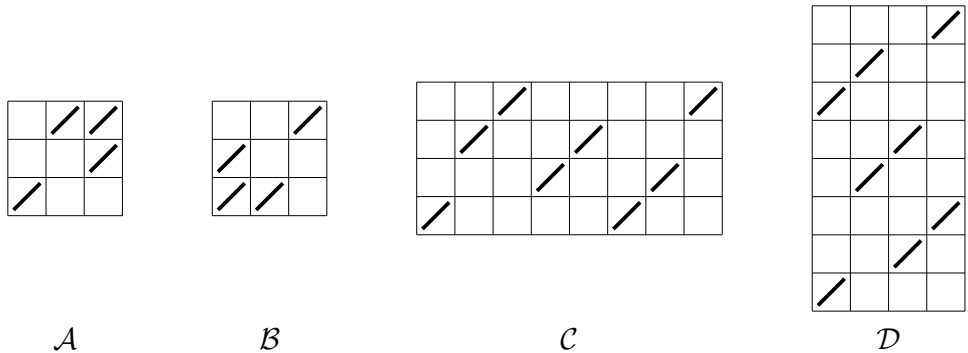
\begin{center}
\begin{tabular}{ccccccc}
\psset{xunit=0.01in, yunit=0.01in} \psset{linewidth=0.005in}
\begin{pspicture}(0,-50)(60,60)
\drawlines(3,3)
\upline(1,1)
\upline(2,3)
\upline(3,2)
\upline(3,3)
\end{pspicture}
&\rule{10pt}{0pt}&
\psset{xunit=0.01in, yunit=0.01in} \psset{linewidth=0.005in}
\begin{pspicture}(0,-50)(60,60)
\drawlines(3,3)
\upline(1,1)
\upline(1,2)
\upline(2,1)
\upline(3,3)
\end{pspicture}
&\rule{10pt}{0pt}&
\psset{xunit=0.01in, yunit=0.01in} \psset{linewidth=0.005in}
\begin{pspicture}(0,-40)(160,80)
\drawlines(8,4)
\upline(1,1)
\upline(2,3)
\upline(3,4)
\upline(4,2)
\upline(5,3)
\upline(6,1)
\upline(7,2)
\upline(8,4)
\end{pspicture}
&\rule{10pt}{0pt}&
\psset{xunit=0.01in, yunit=0.01in} \psset{linewidth=0.005in}
\begin{pspicture}(0,0)(80,160)
\drawlines(4,8)
\upline(1,1)
\upline(3,2)
\upline(4,3)
\upline(2,4)
\upline(3,5)
\upline(1,6)
\upline(2,7)
\upline(4,8)
\end{pspicture}
\\
$\A$&&$\B$&&$\C$&&$\D$
\end{tabular}
\end{center}

Classes $\A$ and $\B$ are self-inverse and each is the reverse-complement of the other,  while classes $\C$ and $\D$ are each equal to their reverse-complement and are the inverses of one another.  The proof uses the following result from \cite{atkinson:restricted-perm:}.
\begin{proposition}\label{Av321-2143}
$\Av(321,2143)$ is the union of two monotone grid classes, namely those whose gridding matrices are
$(1\ 1)$
%$\bigl(\begin{smallmatrix}1&1 \end{smallmatrix} \bigr)$
%\pmatrix$\Grid\left(\begin{array}{cc}1&1\end{array}\right)$
%$[I,I]$
 and its transpose.
\end{proposition}

\begin{lemma}\label{two-descents}
Class $\C$ is exactly the intersection of $\Av(2143, 4321)$ with the class of permutations having at most 2 descents.
\end{lemma}
\begin{proof}
Figure \ref{2descents} shows a permutation with 2 descents.  If it contains a 2143 pattern then there will be one where the 1 of the pattern matches the element $a$ and the 4 matches the element $d$.  The diagram identifies the potential 2's and 3's and, when we impose the condition that no 2 is below a 3, we obtain class $\C$.
\end{proof}

 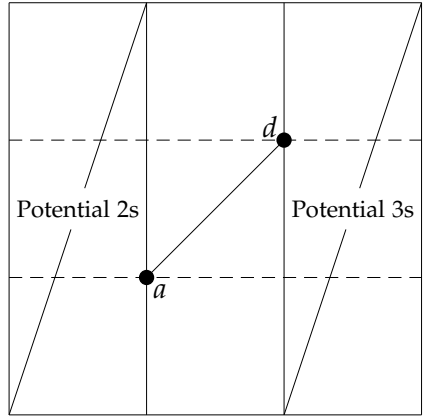
\begin{figure}
\centering
\psset{xunit=0.012in, yunit=0.012in} \psset{linewidth=0.005in}
\begin{pspicture}(0,0)(200,200)
\multido{\i=10+60}{4}{\psline[linestyle=solid,linewidth=0.005in](\i,10)(\i,190)}
\psline[linestyle=solid,linewidth=0.005in](10,10)(190,10)
\psline[linestyle=dashed,linewidth=0.005in](10,70)(190,70)
\psline[linestyle=dashed,linewidth=0.005in](10,130)(190,130)
\psline[linestyle=solid,linewidth=0.005in](10,190)(190,190)
\pscircle*(70,70){0.04in}
\pscircle*(130,130){0.04in}
\rput(76,64){$a$}
\rput(124,136){$d$}
\psline[linestyle=solid,linewidth=0.005in](70,70)(130,130)

\psline[linestyle=solid,linewidth=0.005in](10,10)(37,91)
\psline[linestyle=solid,linewidth=0.005in](43,109)(70,190)

\psline[linestyle=solid,linewidth=0.005in](130,10)(157,91)
\psline[linestyle=solid,linewidth=0.005in](163,109)(190,190)
\rput(40,100){{\footnotesize Potential 2s}}
\rput(160,100){{\footnotesize Potential 3s}}
\end{pspicture}
\caption{$C=\Av(2143, 4321)\cap$ 2 descents}
%\label{caseb4}
\label{2descents}
\end{figure}

Notice that $\A$, $\B$, $\C$ and $\D$ are all contained in $\Av(4321,2143)$ (by inspection,  they cannot contain the patterns $4321$ or $2143$).  So to prove Theorem \ref{maintheorem} it is sufficient to prove that  $\Av(4321,2143)$ is contained in their union.
Let $\pi$ be an arbitrary permutation of $\Av(4321,2143)$ of length $n$.  If $n$ is the final element of $\pi$ then, by induction, $\pi-n$ lies in one of $\A$, $\B$, $\C$ or $\D$ and, from the forms of these classes, so does $\pi$.  So we may assume that there is at least one element to the right  of $n$.  This gives rise to the situation of Figure \ref{barebones} where $m$ is the largest element to the right of $n$.

\begin{figure}
\centering
%Diagram "barebones"
\psset{xunit=0.01in, yunit=0.01in} \psset{linewidth=0.005in}
\begin{pspicture}(0,0)(80,90)
\multido{\i=0+40}{3}{\psline[linestyle=solid,linewidth=0.005in](\i,0)(\i,80)}
\multido{\i=0+40}{3}{\psline[linestyle=solid,linewidth=0.005in](0,\i)(80,\i)}
\pscircle*(40,80){0.04in}
\pscircle*(60,40){0.04in}
\rput(40,90){$n$}
\rput(60,50){$m$}
\psline[linestyle=solid,linewidth=0.02in](3,3)(37,37)
\rput(20,60){$A$}
\rput(60,20){$B$}
\end{pspicture}

\caption{First decomposition}
\label{barebones}
\end{figure}
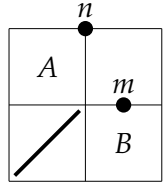

In  Figure \ref{barebones} the bottom-left cell is increasing because $\pi$ avoids $2143$ and, by definition of $m$, the top-right cell is empty.  Also, both cells $A$ and $B$ lie in $\Av(321,2143)$.  Furthermore, at least one of $A,B$ is increasing because of the $4321$-avoidance.  So, by Proposition \ref{Av321-2143}, we get the four cases shown in Figure \ref{InitialCases}.  Of these cases a) and d) are mutually inverse as are cases b) and c), so we shall consider only cases a) and b).

In case a) we begin by noting that if the two increasing cells on the left were a single increasing sequence then Lemma \ref{two-descents} would apply to show $\pi\in\C$.  So we may assume they contain at least one descent.  Now we impose the $4321$-avoidance condition.  By the preceding remark the situation is as shown in Figure \ref{casea1} with $d$ preceding $c$.  If such a permutation had a $4321$ pattern then it would have one where $d$ matched $4$,  $c$ matched $3$, $a$ matched $1$ and all possible 2's would lie between $c$ and $a$ in both position and value.    Since there is no $4321$ pattern, either $a$ is above $c$ or $a$ is below $c$ but the set of possible $2$s is empty giving one of the diagrams in Figure \ref{casea2}.

%Four basic cases of the case analysis
\begin{figure}
\centering
\begin{tabular}{ccccccc}
\psset{xunit=0.008in, yunit=0.008in} \psset{linewidth=0.005in}
\begin{pspicture}(0,0)(120,120)
\multido{\i=0+40}{4}{\psline(\i,0)(\i,80)}
\multido{\i=0+40}{3}{\psline(0,\i)(120,\i)}
\psline[linestyle=solid,linewidth=0.02in](3,3)(37,37)
\psline[linestyle=solid,linewidth=0.02in](3,43)(37,77)
\psline[linestyle=solid,linewidth=0.02in](43,3)(77,37)
\psline[linestyle=solid,linewidth=0.02in](83,3)(117,37)
\end{pspicture}
&\rule{10pt}{0pt}&
\psset{xunit=0.008in, yunit=0.008in} \psset{linewidth=0.005in}
\begin{pspicture}(0,0)(120,120)
\multido{\i=20+40}{3}{\psline(\i,0)(\i,120)}
\psline(20,0)(100,0)
\psline(60,40)(100,40)
\psline(20,80)(100,80)
\psline(20,120)(100,120)
\psline[linestyle=solid,linewidth=0.02in](23,3)(57,77)
\psline[linestyle=solid,linewidth=0.02in](63,3)(97,37)
\psline[linestyle=solid,linewidth=0.02in](63,43)(97,77)
\psline[linestyle=solid,linewidth=0.02in](23,83)(57,117)
\end{pspicture}
&\rule{10pt}{0pt}&
\psset{xunit=0.008in, yunit=0.008in} \psset{linewidth=0.005in}
\begin{pspicture}(0,0)(120,120)
\psline(0,0)(0,80)
\psline(40,40)(40,80)
\psline(80,0)(80,80)
\psline(120,0)(120,80)
\multido{\i=0+40}{3}{\psline(0,\i)(120,\i)}
\psline[linestyle=solid,linewidth=0.02in](3,3)(77,37)
\psline[linestyle=solid,linewidth=0.02in](3,43)(37,77)
\psline[linestyle=solid,linewidth=0.02in](43,43)(77,77)
\psline[linestyle=solid,linewidth=0.02in](83,3)(117,37)
\end{pspicture}
&\rule{10pt}{0pt}&
\psset{xunit=0.008in, yunit=0.008in} \psset{linewidth=0.005in}
\begin{pspicture}(0,0)(120,120)
\multido{\i=20+40}{3}{\psline(\i,0)(\i,120)}
\multido{\i=0+40}{4}{\psline(20,\i)(100,\i)}
\psline[linestyle=solid,linewidth=0.02in](23,3)(57,37)
\psline[linestyle=solid,linewidth=0.02in](23,43)(57,77)
\psline[linestyle=solid,linewidth=0.02in](63,3)(97,37)
\psline[linestyle=solid,linewidth=0.02in](23,83)(57,117)
\end{pspicture}
\\
(a)&&(b)&&(c)&&(d)
\end{tabular}
\caption{Four cases}
\label{InitialCases}
\end{figure}
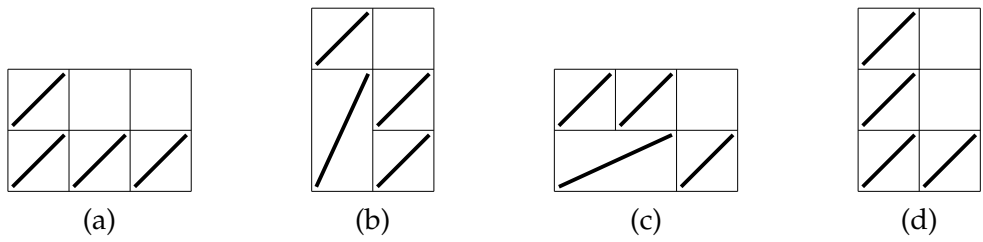

\begin{figure}
\centering
\psset{xunit=0.01in, yunit=0.01in} \psset{linewidth=0.005in}
\begin{pspicture}(0,0)(100,130)
\multido{\i=10+40}{4}{\psline[linestyle=solid,linewidth=0.005in](\i,20)(\i,100)}
\multido{\i=20+40}{3}{\psline[linestyle=solid,linewidth=0.005in](10,\i)(130,\i)}

\psline[linestyle=solid,linewidth=0.02in](17,67)(43,93)
\psline[linestyle=solid,linewidth=0.02in](17,27)(43,53)
\psline[linestyle=solid,linewidth=0.02in](53,23)(87,57)
\psline[linestyle=solid,linewidth=0.02in](97,27)(127,57)
\pscircle*(17,67){0.04in}
\pscircle*(43,53){0.04in}
\pscircle*(97,27){0.04in}
\rput(17,80){$d$}
\rput(43,42){$c$}
\rput(108,26){$a$}
\end{pspicture}
\caption{Case a) possible occurrences of $4321$}
\label{casea1}
\end{figure}
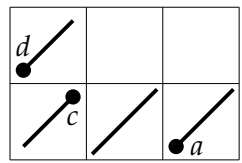

\begin{figure}
\centering
\psset{xunit=0.013in, yunit=0.013in} \psset{linewidth=0.005in}
\begin{pspicture}(0,0)(340,140)
%\multido{\i=10+40}{4}{\psline[linestyle=solid,linewidth=0.005in](\i,220)(\i,300)}
\multido{\i=20+40}{4}{\psline[linestyle=solid,linewidth=0.005in](0,\i)(160,\i)}
\psline[linestyle=solid,linewidth=0.005in](0,20)(0,140)
\psline[linestyle=solid,linewidth=0.005in](40,20)(40,140)
\psline[linestyle=solid,linewidth=0.005in](80,20)(80,100)
\psline[linestyle=solid,linewidth=0.005in](120,60)(120,100)
\psline[linestyle=solid,linewidth=0.005in](160,20)(160,140)

\psline[linestyle=solid,linewidth=0.02in](3,23)(36,56)
\psline[linestyle=solid,linewidth=0.02in](4,104)(37,137)
\psline[linestyle=solid,linewidth=0.02in](43,23)(77,57)
\psline[linestyle=solid,linewidth=0.02in](83,63)(117,97)
\psline[linestyle=solid,linewidth=0.02in](123,63)(157,97)

\pscircle*(36,56){0.04in}
\rput(33,65){$c$}
\pscircle*(4,104){0.04in}
\rput(13,105){$d$}

\multido{\i=20+40}{4}{\psline[linestyle=solid,linewidth=0.005in](180,\i)(340,\i)}
\psline[linestyle=solid,linewidth=0.005in](180,20)(180,140)
\psline[linestyle=solid,linewidth=0.005in](220,20)(220,140)
\psline[linestyle=solid,linewidth=0.005in](240,20)(240,100)
\psline[linestyle=solid,linewidth=0.005in](280,20)(280,100)
\psline[linestyle=solid,linewidth=0.005in](300,20)(300,100)
\psline[linestyle=solid,linewidth=0.005in](340,20)(340,140)
\psline[linestyle=solid,linewidth=0.005in](220,40)(340,40)

\psline[linestyle=solid,linewidth=0.02in](183,23)(216,56)
\psline[linestyle=solid,linewidth=0.02in](183,103)(220,140)
\psline[linestyle=solid,linewidth=0.02in](243,63)(277,97)
\psline[linestyle=solid,linewidth=0.02in](303,63)(337,97)
\psline[linestyle=solid,linewidth=0.02in](223,23)(237,37)
\psline[linestyle=solid,linewidth=0.02in](284,44)(297,57)

\pscircle*(216,56){0.04in}
\rput(213,65){$c$}
\pscircle*(183,103){0.04in}
\rput(193,105){$d$}
\pscircle*(284,44){0.04in}
\rput(292,45){$a$}
\rput(236, 27){$I_1$}
\rput(265, 72){$I_2$}

\end{pspicture}
\caption{Case a) with $4321$-avoidance}
\label{casea2}
\end{figure}
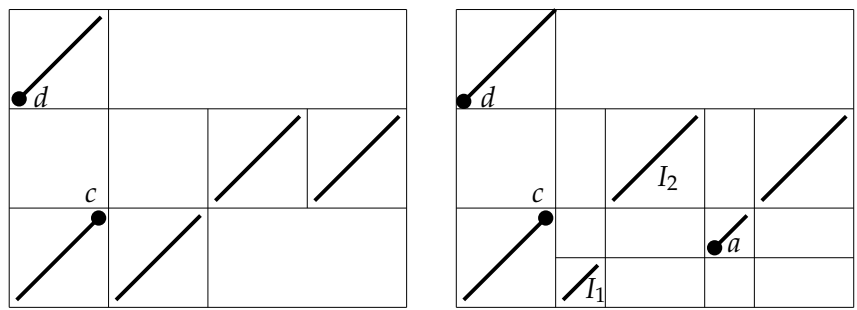

%Also, for the case mentioned at the beginning of case a) we have the possible diagram (Figure \ref{casea3})
%
%\begin{figure}
%\begin{center}
%\includegraphics[width=2in]{chase-in-text-2b}
%\caption{Case a): avoiding $4321$ yet again}
%\label{casea3}
%\end{center}
%\end{figure}

To conclude the analysis of case a) we have to impose the $2143$-avoidance conditions on the two diagrams of Figure \ref{casea2}.  In the case of the first diagram if there is a $2143$ pattern then there would be one where $c$ matches $2$.  So to avoid 2143  either the increasing cell below and to the right of $c$ must be empty or the two rightmost increasing cells have to contain a single increasing sequence.  It then follows by inspection that $\pi\in\B$.

%\begin{figure}
%\begin{center}
%\includegraphics[width=3in]{chase-in-text-2aextra}
%\caption{Case a): avoiding $2143$}
%\label{extra}
%\end{center}
%\end{figure}
%
For the second diagram of Figure \ref{casea2} the argument is a little more complicated.  The only way that 2143 could be contained as a pattern would be for the 1 to lie in cell $I_1$ and the 4 to lie in cell $I_2$.  If either cell $I_1$ or $I_2$ is empty then $\pi\in\B$ so we shall now assume both are non-empty.  If there is a  $2143$ pattern at all there would be one where the 1 is the lowest element of $I_1$ and the 4 is the highest element of $I_2$. We now consider a more detailed diagram (Figure \ref{nonemptyI1I2}) where these two points are displayed as  black circles.

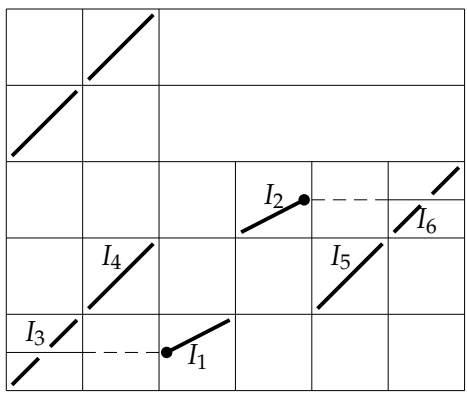
\begin{figure}
\centering
\psset{xunit=0.01in, yunit=0.01in} \psset{linewidth=0.005in}
\begin{pspicture}(0,0)(220,240)
\multido{\i=20+40}{6}{\psline[linestyle=solid,linewidth=0.005in](0,\i)(240,\i)}
\psline[linestyle=solid,linewidth=0.005in](0,20)(0,220)
\psline[linestyle=solid,linewidth=0.005in](40,20)(40,220)
\psline[linestyle=solid,linewidth=0.005in](80,20)(80,220)
\psline[linestyle=solid,linewidth=0.005in](120,20)(120,140)
\psline[linestyle=solid,linewidth=0.005in](160,20)(160,140)
\psline[linestyle=solid,linewidth=0.005in](200,20)(200,140)
\psline[linestyle=solid,linewidth=0.005in](240,20)(240,220)
\psline[linestyle=dashed,linewidth=0.005in](160,120)(200,120)
\psline[linestyle=solid,linewidth=0.005in](200,120)(240,120)
\psline[linestyle=solid,linewidth=0.005in](0,40)(40,40)
\psline[linestyle=dashed,linewidth=0.005in](40,40)(80,40)

%\psline[linestyle=dashed,linewidth=0.005in](20,20)(20,60)
%\psline[linestyle=dashed,linewidth=0.005in](220,100)(220,140)

\psline[linestyle=solid,linewidth=0.02in](3,23)(17,37)
\psline[linestyle=solid,linewidth=0.02in](23,43)(37,57)
\psline[linestyle=solid,linewidth=0.02in](43,63)(77,97)

\psline[linestyle=solid,linewidth=0.02in](3,143)(37,177)
\psline[linestyle=solid,linewidth=0.02in](43,183)(77,217)
\psline[linestyle=solid,linewidth=0.02in](84,40)(117,57)
\psline[linestyle=solid,linewidth=0.02in](123,103)(156,120)
\psline[linestyle=solid,linewidth=0.02in](163,63)(197,97)
\psline[linestyle=solid,linewidth=0.02in](203,103)(217,117)
\psline[linestyle=solid,linewidth=0.02in](223,123)(237,137)

\pscircle*(84,40){0.03in}
\pscircle*(156,120){0.03in}

\rput(100,38){$I_1$}
\rput(140,122){$I_2$}
\rput(15,50){$I_3$}
\rput(55,90){$I_4$}
\rput(175,90){$I_5$}
\rput(220,109){$I_6$}

\end{pspicture}
\caption{Case a): non-empty $I_1$ and $I_2$}
\label{nonemptyI1I2}
\end{figure}

In order to ensure that the pattern 2143 is avoided in Figure 6, some restrictions must be imposed. Specifically, there are three possibilities:

\begin{itemize}
\item $I_6$ is non-empty and therefore both $I_3$ and $I_4$ are empty.  Then $\pi\in \B$.
\item $I_3$ is non-empty and therefore both $I_5$ and $I_6$ are empty.  Again $\pi\in \B$.
\item $I_3$ and $I_6$ are both empty and  all points of $I_4$ are larger than all points of $I_5$.  Then $\pi$ is the inverse of a permutation with at most 2 descents and so, by Lemma \ref{two-descents}, lies in class $\D$.
\end{itemize}

Now we turn to case b) in Figure \ref{InitialCases}.  Again we begin by imposing the $4321$-avoidance condition.  This is automatic if the two leftmost increasing cells form a single increasing region.  If they do not and there were a $4321$ pattern at all then we could take one where the point matching the $3$ was the top point $c$ of the lower left increasing cell; then the points $2$ and $1$ would have to be furnished from the two rightmost increasing cells and so the higher of these would be greater than $c$.  Therefore $\pi$ would have one of the two forms shown in Figure \ref{b-cases}.

\begin{figure}
\centering
\psset{xunit=0.01in, yunit=0.01in} \psset{linewidth=0.005in}
\begin{pspicture}(0,0)(240,140)
\psline[linestyle=solid,linewidth=0.005in](0,20)(0,100)
\psline[linestyle=solid,linewidth=0.005in](40,20)(40,100)
\psline[linestyle=solid,linewidth=0.005in](80,20)(80,100)
\psline[linestyle=solid,linewidth=0.005in](120,20)(120,100)
\psline[linestyle=solid,linewidth=0.005in](0,20)(120,20)
\psline[linestyle=solid,linewidth=0.005in](0,60)(120,60)
\psline[linestyle=solid,linewidth=0.005in](0,100)(120,100)

\psline[linestyle=solid,linewidth=0.02in](3,23)(37,57)
\psline[linestyle=solid,linewidth=0.02in](43,63)(77,97)
\psline[linestyle=solid,linewidth=0.02in](83,63)(117,97)
\psline[linestyle=solid,linewidth=0.02in](84,24)(116,56)
\pscircle*(84,24){0.04in}
\pscircle*(116,56){0.04in}

\psline[linestyle=solid,linewidth=0.005in](160,20)(160,140)
\psline[linestyle=solid,linewidth=0.005in](200,20)(200,140)
\psline[linestyle=solid,linewidth=0.005in](240,20)(240,140)
\psline[linestyle=solid,linewidth=0.005in](160,20)(240,20)
\psline[linestyle=solid,linewidth=0.005in](160,60)(240,60)
\psline[linestyle=solid,linewidth=0.005in](160,100)(240,100)
\psline[linestyle=solid,linewidth=0.005in](160,140)(240,140)

\psline[linestyle=solid,linewidth=0.02in](163,23)(196,56)
\psline[linestyle=solid,linewidth=0.02in](163,103)(197,137)
\psline[linestyle=solid,linewidth=0.02in](204,24)(236,56)
\psline[linestyle=solid,linewidth=0.02in](203,63)(237,97)

\pscircle*(204,24){0.04in}
\pscircle*(236,56){0.04in}
\pscircle*(196,56){0.04in}
\rput(190,65){$c$}

\end{pspicture}
\caption{Case b): avoiding $4321$}
\label{b-cases}
\end{figure}
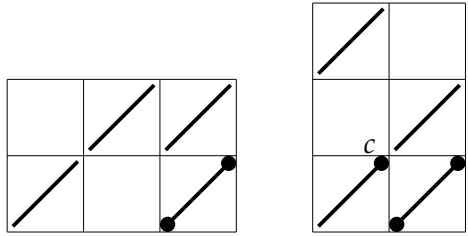

%\begin{figure}
%\begin{center}
%\includegraphics[width=2in]{chase-in-text-5}
%\caption{Case b): avoiding $4321$}
%\label{caseb}
%\end{center}
%\end{figure}

Finally we impose the 2143-avoidance condition.  If there were a 2143 pattern then, in both of the diagrams in Figure \ref{b-cases}, the two marked points of the lower right cell could be taken as the 1 and 3 points.  So there would be a 2143 pattern if and only if there were a separating 2 to the left together with a separating 4 above these marked points.  In other words, because there is no 2143 pattern, one of the two dotted regions in each of the two diagrams of Figure \ref{b-cases-2} must be empty.

\begin{figure}
\begin{center}
\psset{xunit=0.01in, yunit=0.01in} \psset{linewidth=0.005in}
\begin{pspicture}(0,0)(240,140)
\psline[linestyle=solid,linewidth=0.005in](0,20)(0,100)
\psline[linestyle=solid,linewidth=0.005in](40,20)(40,100)
\psline[linestyle=solid,linewidth=0.005in](80,20)(80,100)
\psline[linestyle=solid,linewidth=0.005in](120,20)(120,100)
\psline[linestyle=solid,linewidth=0.005in](0,20)(120,20)
\psline[linestyle=solid,linewidth=0.005in](0,60)(120,60)
\psline[linestyle=solid,linewidth=0.005in](0,100)(120,100)

\psline[linestyle=solid,linewidth=0.02in](3,23)(13,33)
\psline[linestyle=dotted,linewidth=0.02in](13,33)(27,47)
\psline[linestyle=solid,linewidth=0.02in](27,47)(37,57)
\psline[linestyle=solid,linewidth=0.02in](43,63)(77,97)
\psline[linestyle=solid,linewidth=0.02in](83,63)(93,73)
\psline[linestyle=dotted,linewidth=0.02in](93,73)(107,87)
\psline[linestyle=solid,linewidth=0.02in](107,87)(117,97)
\psline[linestyle=solid,linewidth=0.02in](93,33)(107,47)
\pscircle*(93,33){0.04in}
\pscircle*(107,47){0.04in}

\psline[linestyle=solid,linewidth=0.005in](160,20)(160,140)
\psline[linestyle=solid,linewidth=0.005in](200,20)(200,140)
\psline[linestyle=solid,linewidth=0.005in](240,20)(240,140)
\psline[linestyle=solid,linewidth=0.005in](160,20)(240,20)
\psline[linestyle=solid,linewidth=0.005in](160,60)(240,60)
\psline[linestyle=solid,linewidth=0.005in](160,100)(240,100)
\psline[linestyle=solid,linewidth=0.005in](160,140)(240,140)

\psline[linestyle=solid,linewidth=0.02in](163,23)(173,33)
\psline[linestyle=dotted,linewidth=0.02in](173,33)(187,47)
\psline[linestyle=solid,linewidth=0.02in](187,47)(197,57)
\psline[linestyle=solid,linewidth=0.02in](163,103)(197,137)
\psline[linestyle=solid,linewidth=0.02in](213,33)(227,47)
\psline[linestyle=solid,linewidth=0.02in](203,63)(213,73)
\psline[linestyle=dotted,linewidth=0.02in](213,73)(227,87)
\psline[linestyle=solid,linewidth=0.02in](227,87)(237,97)

\pscircle*(213,33){0.04in}
\pscircle*(227,47){0.04in}
%\pscircle*(197,57){0.04in}
%\rput(192,57){c}

\psline[linestyle=dashed,linewidth=0.005in](0,33)(93,33)
\psline[linestyle=dashed,linewidth=0.005in](0,47)(107,47)
\psline[linestyle=dashed,linewidth=0.005in](93,33)(93,100)
\psline[linestyle=dashed,linewidth=0.005in](107,47)(107,100)

\psline[linestyle=dashed,linewidth=0.005in](160,33)(213,33)
\psline[linestyle=dashed,linewidth=0.005in](160,47)(227,47)
\psline[linestyle=dashed,linewidth=0.005in](213,33)(213,140)
\psline[linestyle=dashed,linewidth=0.005in](227,47)(227,140)

\end{pspicture}
\caption{Case b): avoiding $2143$}
\label{b-cases-2}
\end{center}
\end{figure}
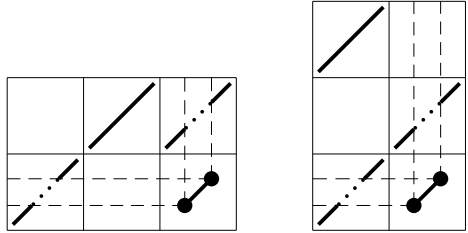

Then, by inspection, we have the following conclusions:

\begin{itemize}
\item If the dotted region in the bottom left cell of the left diagram is empty $\pi\in \A$.
\item If the dotted region in the top right cell of the left diagram is empty $\pi\in \C$.
\item If the dotted region in the bottom left cell of the right diagram is empty $\pi\in \D$.
\item If the dotted region in the middle right cell of the right diagram is empty $\pi\in \B$.
\end{itemize}

In all cases, $\pi$ lies in one of $\A$, $\B$, $\C$ or $\D$.

%{\bf Notes}
%
%It would be nice to have the bases of $A$, $B$, $C$ and $D$.  The basis of $A$ may be
%
%\[4321, 2143, 41532, 25413, 42531, 52413, 42513\]
%
%There is some evidence from PermutationLab for this.
%
%For class $C$ how about first imposing ``no 3 descents''?  42531
%53142
%41532
%43152
%415263.
% This is the basis!  $C$ is certainly contained in the class of all permutations with at most 2 descents.  In this class we impose the 2143 condition - an easy diagram chase results in $C$.

\subsection{The enumeration of Av(2143, 4321)}

To enumerate the class $\Av(2143, 4321)$ we will enumerate the individual classes $\A$, $\B$, $\C$, $\D$, and then their various intersections with one another so that the final result can be computed by inclusion-exclusion. This will be accomplished by first finding regular languages in one to one correspondence with each of them. As noted earlier, the existence of such languages follows from general results of \cite{albert:geometric-grid-:}, but the classes considered here are just simple enough to be handled directly and more easily by ad hoc methods.  Finding regular languages that correspond to their intersections is then quite straightforward provided that we know the bases of each one. This is because the correspondence between forest grid classes and words is such that subpermutations correspond to subwords, so to find a regular language corresponding to the intersection of a class $\X$ and $\Av(F)$ for some finite set of permutations $F$, it suffices to take the intersection of a regular language for $\X$ with the complement of the language consisting of those words $w$ that contain a subword $v$ which encodes an element of $F$ (this language is easily seen to be regular). Finally, each of the regular languages can be enumerated by standard variations of the transfer matrix technique.

We first exhibit (and justify) the basis for $\B$:
\[
\B = \Av(2143, 4321, 35142, 35214, 35241,  43152, 53142).
\]

It is routine to check that none of the permutations on the right hand side belong to $\B$, so to verify this result we must show that any permutation that avoids them all belongs to $\B$. Suppose that this were false and take a minimal counterexample $\pi$. Then $\pi$ does not end with its maximum since $\B$ is closed under ``append a new maximum''. Furthermore, $\pi$ contains $321$ since the characterisation of $\Av(321, 2143)$ given in Proposition \ref{Av321-2143} shows that this class is a subclass of $\B$. Take a particular copy, $cba$ of $321$ in $\pi$ chosen so that $b$ is the largest $2$ of any $321$, while $a$ and $c$ are the smallest 1's and 3's respectively that form a 321 with $b$.

Now a relatively simple case analysis will be used to show that in fact $\pi \in \B$, thus obtaining a contradiction. We first consider the structure of $\pi$ to the right of $a$. Here, there can be no element $d$, larger than $c$. For if there were such an element, all $e > d$ would have to follow $d$ (else one of  4321, 35214, or 2143 would occur in $\pi$) in increasing order (else 2143), and so $\pi$ would end with its maximum. Furthermore, there can be no element smaller than $a$ (else 4321), and elements intermediate between $a$ and $c$ must occur in increasing order (else 4321, 53142, or 2143). Now we turn our attention to the left of $a$. By the choices of $c$ and $b$, and 4321-avoidance, there are no elements here whose value is intermediate between $b$ and $c$. Likewise by the choices of $b$ and $a$ and 4321-avoidance, there are no elements whose position lies between $b$ and $a$ and whose value lies below $c$. Now, in a similar fashion as above, it is easy to argue that the elements to the left of and below $b$ occur in increasing order, as do those to the right of and above $c$. Thus $\pi \in \B$ (in fact in the part of $\B$ in the lower left $2 \times 2$ subgrid) as witnessed by a horizontal grid line just below $c$ and a vertical one just to the left of $a$.

This analysis also shows how to construct a regular language in one to one correspondence with $\B$. We will use an alphabet $\{ \a, \b, \c, \d \}$. These will correspond to the cells of the gridding with $\d$ being the upper right, $\b$ the lower left, $\a$ the second cell in the bottom row, and $\c$ the second cell in the left hand column. The reading order in each cell is left to right and hence bottom to top. We now describe how to associate each $\pi \in \B$ with a unique word over this language (and implicitly claim that the set of such words forms a regular language -- which will hopefully be clear from the construction). If $\pi$ does not involve $321$ then either it belongs to the class whose gridding matrix is $(1 \ 1)$ or its transpose (or both). If it belongs to the first of these but not the second, then it can be uniquely represented over the alphabet $\{\a, \b\}$, while in the opposite case it can be uniquely represented over $\{\b, \c\}$. If it belongs to the intersection and is not increasing, we can choose its representation to be of the form $\a^k \b^l$, while if it is increasing, we will simply use $\b^n$. If $\pi$ does involve $321$ then the argument of the preceding paragraph shows that we \emph{must} encode the largest such 2 with $\b$, the smallest corresponding $3$ with $\c$ and the smallest corresponding $1$ with $\a$. Moreover, all the letters encoding the remaining elements of $\pi$ are  uniquely determined by their position relative to this 321. The only latitude allowed in the arrangement of these letters is that adjacent $\c$ and $\a$ elements could be placed in either order, and also that only the total number of $\d$'s is relevant. Finally, to ensure the existence of a 321, there must be both a $\c$ and an $\a$ preceding some $\b$. Thus we can obtain a unique representative for each permutation involving 321 in this class by requiring that the encoding word be of the form $\{\a, \b, \c\}^* \d^*$, that it must contain the substrings $\c \b$ and  $\a \b $, and that it may not contain $\c \a$ as a factor.

Class $\A$ is a symmetry of $\B$ so we need consider it no further. For class $\C$ (and symmetrically $\D$) a basis is easy to determine since it is the subclass of $\Av(2143, 4321)$ allowing at most two descents, and so its basis consists of the minimal permutations having at most three descents (and not involving the two known basis elements). The regular encoding of $\C$ is also easily obtained by thinking of it as a subclass of the grid class with gridding matrix $(1 \ 1 \ 1)$ and starting from a regular encoding of that class.

With these bases and regular languages computed, the remainder of the program outlined in the first paragraph of this subsection can be carried out. In practice, all of these constructions were accomplished using GAP \cite{GAP4} and its automata package \cite{GAP_Automata}. Of course a signal advantage of this approach is that it is possible to verify all the constructions experimentally (at least up to a certain length) thereby obtaining some not insignificant degree of confidence in their correctness. In view of this procedure, we report only the final outcome of the enumeration.

\begin{proposition}
The class $\Av(2143, 4321)$ is enumerated by the rational function:
 \[
 \frac{t p(t)}{(1-2
t)^4 (1-t)^7 \left(1-3t+t^2\right)}
\]
where
%\begin{eqnarray*}
%p(t) &=& -16 t^{12}+150 t^{11}-690 t^{10}+1968 t^9-3691 t^8+4827
%t^7-  \\
%{} & {} &   4530 t^6 + 3074 t^5-1499 t^4+513 t^3-117 t^2+16 t-1 .
%\end{eqnarray*}
\begin{eqnarray*}
p(t) &=& 1 - 16 t + 117 t^2 - 513 t^3 + 1499 t^4 - 3064 t^5 + 4530 t^6 -   \\
{} & {} &   4827 t^7 + 3691 t^8 - 1968 t^9 + 690 t^{10} - 150 t^{11} + 16 t^{12} .
\end{eqnarray*}

\end{proposition}

The series expansion begins 1, 2, 6, 22, 86, 333, 1235, 4339, 14443, 45770, 138988, 407134, $\ldots$ (sequence A165525 of~\cite{sloane:the-on-line-enc:}).

\section{Av(2143, 4312)}
\subsection{The structure of Av(2143,4312)}

The main result in this subsection is

\begin{theorem}\label{Av-2143-4312-structure}
\[\Av(2143, 4312)=\E\cup\F\]
where $\E$ and $\F$ are the monotone grid classes shown in Figure \ref{E-and-F}.
\end{theorem}

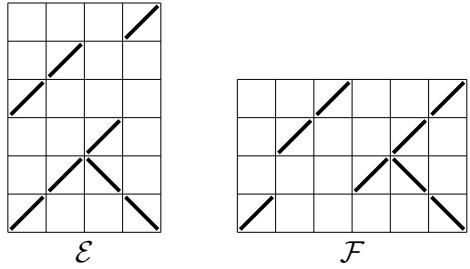
\begin{figure}
\centering
\psset{xunit=0.005in, yunit=0.005in} \psset{linewidth=0.005in}
\begin{pspicture}(0,0)(480,260)
\psline[linestyle=solid,linewidth=0.005in](0,20)(160,20)
\psline[linestyle=solid,linewidth=0.005in](0,60)(160,60)
\psline[linestyle=solid,linewidth=0.005in](0,100)(160,100)
\psline[linestyle=solid,linewidth=0.005in](0,140)(160,140)
\psline[linestyle=solid,linewidth=0.005in](0,180)(160,180)
\psline[linestyle=solid,linewidth=0.005in](0,220)(160,220)
\psline[linestyle=solid,linewidth=0.005in](0,260)(160,260)
\psline[linestyle=solid,linewidth=0.005in](0,20)(0,260)
\psline[linestyle=solid,linewidth=0.005in](40,20)(40,260)
\psline[linestyle=solid,linewidth=0.005in](80,20)(80,260)
\psline[linestyle=solid,linewidth=0.005in](120,20)(120,260)
\psline[linestyle=solid,linewidth=0.005in](160,20)(160,260)

\psline[linestyle=solid,linewidth=0.02in](3,23)(37,57)
\psline[linestyle=solid,linewidth=0.02in](3,143)(37,177)
\psline[linestyle=solid,linewidth=0.02in](43,63)(77,97)
\psline[linestyle=solid,linewidth=0.02in](43,183)(77,217)
\psline[linestyle=solid,linewidth=0.02in](83,97)(117,63)
\psline[linestyle=solid,linewidth=0.02in](83,103)(117,137)
\psline[linestyle=solid,linewidth=0.02in](123,57)(157,23)
\psline[linestyle=solid,linewidth=0.02in](123,223)(157,257)
\rput(80,00){$\E$}
\psline[linestyle=solid,linewidth=0.005in](240,20)(480,20)
\psline[linestyle=solid,linewidth=0.005in](240,60)(480,60)
\psline[linestyle=solid,linewidth=0.005in](240,100)(480,100)
\psline[linestyle=solid,linewidth=0.005in](240,140)(480,140)
\psline[linestyle=solid,linewidth=0.005in](240,180)(480,180)
\psline[linestyle=solid,linewidth=0.005in](240,20)(240,180)
\psline[linestyle=solid,linewidth=0.005in](280,20)(280,180)
\psline[linestyle=solid,linewidth=0.005in](320,20)(320,180)
\psline[linestyle=solid,linewidth=0.005in](360,20)(360,180)
\psline[linestyle=solid,linewidth=0.005in](400,20)(400,180)
\psline[linestyle=solid,linewidth=0.005in](440,20)(440,180)
\psline[linestyle=solid,linewidth=0.005in](480,20)(480,180)
\psline[linestyle=solid,linewidth=0.02in](323,143)(357,177)
\psline[linestyle=solid,linewidth=0.02in](443,143)(477,177)
\psline[linestyle=solid,linewidth=0.02in](283,103)(317,137)
\psline[linestyle=solid,linewidth=0.02in](403,103)(437,137)
\psline[linestyle=solid,linewidth=0.02in](363,63)(397,97)
\psline[linestyle=solid,linewidth=0.02in](403,97)(437,63)
\psline[linestyle=solid,linewidth=0.02in](243,23)(277,57)
\psline[linestyle=solid,linewidth=0.02in](443,57)(477,23)
\rput(360,0){$\F$}
\end{pspicture}
\caption{The components of $\Av(2143, 4312)$. Symmetries under reverse-inverse-reverse.}
\label{E-and-F}
\end{figure}

Our first step is to show that $\Av(2143, 4312)$ is contained within a certain $2\times 2$ grid class:

\begin{lemma}\label{lem-two-by-two} Every permutation of $\Av(2143,4312)$ has a diagram of the type shown in Figure \ref{GridD}.
\end{lemma}

\begin{proof}
Consider a permutation $\pi\in\Av(2143,4312)$ of length $n$. We may suppose that $\pi\not\in\Av(132,312)$ (which is the class of wedge permutations oriented $<$), as otherwise the conclusion follows immediately. Let $i$ be the rightmost element which occurs either as the 1 in a $132$ or the $3$ in a 312 pattern. Thus the points that follow $i$ form a pattern in $\Av(132,312)$. There are two cases, and these are depicted in Figure~\ref{fig-two-by-two}. In both cases, the vertical dotted line passes through the element $i$, with the point as labelled. The diagram on the left depicts the case where we find a copy of 132, noting that we have taken the `3' and `2' to be the leftmost possible.  To the left of the element $i$, we divide the region into 3 boxes: the highest (above the `2' of the 132 pattern) must be increasing to avoid 4312, the middle region (below the `2' but above the '1') must be empty to avoid 2143, and the lowest region (below the `1') must be increasing, also to avoid 2143. It is then clear that any permutation in this picture can be represented as stated in the lemma. A similar argument applies in the case when we encounter a 312 pattern first.
\end{proof}

\begin{figure}
\centering
\begin{tabular}{ccc}
\psset{xunit=0.008in, yunit=0.008in} \psset{linewidth=0.005in}
\begin{pspicture}(0,0)(80,80)
\psline[linestyle=solid,linewidth=0.005in](0,0)(80,0)
\psline[linestyle=solid,linewidth=0.005in](0,40)(80,40)
\psline[linestyle=solid,linewidth=0.005in](0,80)(80,80)
\psline[linestyle=solid,linewidth=0.005in](0,0)(0,80)
\psline[linestyle=solid,linewidth=0.005in](40,0)(40,80)
\psline[linestyle=solid,linewidth=0.005in](80,0)(80,80)
\psline[linestyle=solid,linewidth=0.02in](3,3)(37,37)
\psline[linestyle=solid,linewidth=0.02in](3,43)(37,77)
\psline[linestyle=solid,linewidth=0.02in](43,43)(77,77)
\psline[linestyle=solid,linewidth=0.02in](43,37)(77,3)
\end{pspicture}
\end{tabular}
\caption{A $2\times 2$ grid class containing $\Av(2143, 4312)$.}
\label{GridD}
\end{figure}
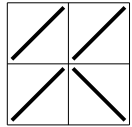

\begin{figure}
\centering
%\begin{tabular}{ccc}
\psset{xunit=0.012in, yunit=0.012in} \psset{linewidth=0.005in}
\begin{pspicture}(0,0)(240,100)
\psline[linestyle=solid,linewidth=0.005in](0,0)(80,0)
\psline[linestyle=solid,linewidth=0.005in](0,100)(80,100)
\psline[linestyle=dashed,linewidth=0.005in](0,40)(40,40)
\psline[linestyle=dashed,linewidth=0.005in](0,60)(80,60)
\psline[linestyle=solid,linewidth=0.005in](0,0)(0,100)
\psline[linestyle=solid,linewidth=0.005in](80,0)(80,100)
\psline[linestyle=dashed,linewidth=0.005in](40,0)(40,100)
\psline[linestyle=solid,linewidth=0.02in](3,3)(37,37)
\psline[linestyle=solid,linewidth=0.02in](5,65)(37,97)
\psline[linestyle=solid,linewidth=0.02in](45,70)(77,97)
\psline[linestyle=solid,linewidth=0.02in](50,60)(77,3)
\pscircle*(40,40){0.04in}
\pscircle*(45,70){0.04in}
\pscircle*(50,60){0.04in}
\rput(46,38){$i$}
\psline[linestyle=solid,linewidth=0.005in](160,0)(240,0)
\psline[linestyle=solid,linewidth=0.005in](160,100)(240,100)
\psline[linestyle=solid,linewidth=0.005in](160,0)(160,100)
\psline[linestyle=solid,linewidth=0.005in](240,0)(240,100)
\psline[linestyle=dashed,linewidth=0.005in](160,40)(240,40)
\psline[linestyle=dashed,linewidth=0.005in](160,80)(200,80)
\psline[linestyle=dashed,linewidth=0.005in](200,0)(200,100)
\psline[linestyle=solid,linewidth=0.02in](163,3)(197,37)
\psline[linestyle=solid,linewidth=0.02in](163,43)(197,77)
\psline[linestyle=solid,linewidth=0.02in](205,30)(237,3)
\psline[linestyle=solid,linewidth=0.02in](210,40)(237,97)
\pscircle*(200,80){0.04in}
\pscircle*(205,30){0.04in}
\pscircle*(210,40){0.04in}
\rput(206,80){$i$}
\end{pspicture}
%\end{tabular}
\caption{The two cases in the proof of Lemma~\ref{lem-two-by-two}.}
\label{fig-two-by-two}
\end{figure}
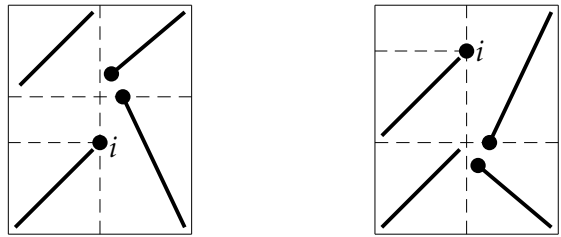

We can now complete the proof of Theorem \ref{Av-2143-4312-structure}.

Certainly both $\E$ and $\F$ are contained in $\Av(2143, 4312)$ since, by inspection, neither contain $2143$ or $4312$.  So, let $\pi$ be an arbitrary permutation of $\Av(2143, 4312)$.
By Lemma~\ref{lem-two-by-two}, $\pi$ has a representation of the type displayed in Figure \ref{GridD}.  Permutations with such a representation automatically avoid 4312 so it remains to enforce the condition that $\pi$ should also avoid 2143. However the only way in which 2143 can appear in such a permutation is for the 2 and the 4 to lie in the top left cell, the 1 in the bottom left cell and the 3 in the top right cell. Thus, avoiding the pattern 2143 corresponds to the following restriction within the $2\times 2$ gridding: any two points in the top left cell cannot simultaneously be separated by points in the top right and bottom left cells.

Hence $\pi$ has one of the two forms shown in Figure \ref{GridDTwoForms}.  By inspection the left-most diagram lies in $\E$ and the right-most diagram lies in $\F$.

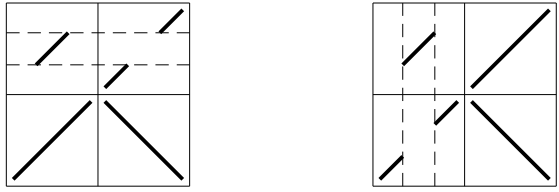
\begin{figure}
\centering
\psset{xunit=0.012in, yunit=0.012in} \psset{linewidth=0.005in}
\begin{pspicture}(0,0)(240,80)
\psline[linestyle=solid,linewidth=0.005in](0,0)(80,0)
\psline[linestyle=solid,linewidth=0.005in](0,40)(80,40)
\psline[linestyle=solid,linewidth=0.005in](0,80)(80,80)
\psline[linestyle=solid,linewidth=0.005in](0,0)(0,80)
\psline[linestyle=solid,linewidth=0.005in](40,0)(40,80)
\psline[linestyle=solid,linewidth=0.005in](80,0)(80,80)
\psline[linestyle=dashed,linewidth=0.005in](0,53)(80,53)
\psline[linestyle=dashed,linewidth=0.005in](0,67)(80,67)

\psline[linestyle=solid,linewidth=0.02in](3,3)(37,37)
\psline[linestyle=solid,linewidth=0.02in](13,53)(27,67)
\psline[linestyle=solid,linewidth=0.02in](43,37)(77,3)
\psline[linestyle=solid,linewidth=0.02in](43,43)(53,53)
\psline[linestyle=solid,linewidth=0.02in](67,67)(77,77)

\psline[linestyle=solid,linewidth=0.005in](160,0)(240,0)
\psline[linestyle=solid,linewidth=0.005in](160,40)(240,40)
\psline[linestyle=solid,linewidth=0.005in](160,80)(240,80)
\psline[linestyle=solid,linewidth=0.005in](160,0)(160,80)
\psline[linestyle=solid,linewidth=0.005in](200,0)(200,80)
\psline[linestyle=solid,linewidth=0.005in](240,0)(240,80)
\psline[linestyle=dashed,linewidth=0.005in](173,0)(173,80)
\psline[linestyle=dashed,linewidth=0.005in](187,0)(187,80)

\psline[linestyle=solid,linewidth=0.02in](163,3)(173,13)
\psline[linestyle=solid,linewidth=0.02in](187,27)(197,37)
\psline[linestyle=solid,linewidth=0.02in](203,37)(237,3)
\psline[linestyle=solid,linewidth=0.02in](203,43)(237,77)
\psline[linestyle=solid,linewidth=0.02in](173,53)(187,67)

\end{pspicture}
\caption{Imposing the 2143-avoidance on Figure \ref{GridD}.}
\label{GridDTwoForms}
\end{figure}

\subsection{Enumerating the simple permutations}

Now we use Theorem \ref{Av-2143-4312-structure}  to enumerate the simple permutations. We first enumerate the simple permutations of the class $\E$. These, together with the simples of $\F$ (whose enumeration is identical), form the simples of the class.

\begin{lemma} The generating function for the simple permutations of length at least 4 in the class $\E$ is
\[s_\E(x) = \frac{x^4(2-3x-x^2)}{(1+x)(1-2x)^2}.\]
\end{lemma}

\begin{proof} Since $\E$ is a forest monotone grid class, we can encode the points of each permutation according to the region in which they lie. We will use the following scheme for our encoding:
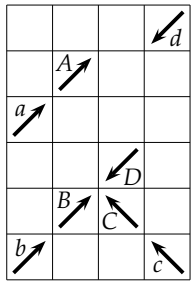
\begin{center}
\psset{xunit=0.012in, yunit=0.012in} \psset{linewidth=0.005in}
\begin{pspicture}(0,0)(80,120)
\drawlines(4,6)
\psline[linestyle=solid,linewidth=0.02in]{->}(3,3)(17,17)
\psline[linestyle=solid,linewidth=0.02in]{->}(23,23)(37,37)
\psline[linestyle=solid,linewidth=0.02in]{->}(3,63)(17,77)
\psline[linestyle=solid,linewidth=0.02in]{->}(23,83)(37,97)
\psline[linestyle=solid,linewidth=0.02in]{<-}(43,43)(57,57)
\psline[linestyle=solid,linewidth=0.02in]{<-}(63,103)(77,117)
\psline[linestyle=solid,linewidth=0.02in]{<-}(43,37)(57,23)
\psline[linestyle=solid,linewidth=0.02in]{<-}(63,17)(77,3)
\rput(6,14){\footnotesize $b$}
\rput(25,35){\footnotesize $B$}
\rput(6,74){\footnotesize $a$}
\rput(25,95){\footnotesize $A$}
\rput(55,45){\footnotesize $D$}
\rput(74,106){\footnotesize $d$}
\rput(45,26){\footnotesize $C$}
\rput(66,6){\footnotesize $c$}
\end{pspicture}
\end{center}
We impose rules on the permitted code words to ensure both that the permutation being encoded is simple and that each simple permutation has a unique encoding. To do this, we specify the following rules:

\begin{itemize}
\item All occurrences of $\{a,b,c,d\}$ come before $\{A,B,C,D\}$.
\item To ensure the permutation is not sum- or skew-decomposable, the word cannot begin with any of $\{b,c,d,B,C\}$. For the same reason, we must find at least one $c$ before the first $d$, at least one $b$ before the first $c$, and an $a$ before the first $b$. Note that this means the first lowercase letter (if it exists) is $a$.
\item All letters before the first $D$ are in $\{a,b,c,d\}$, to avoid duplicate encodings (e.g.\ $A$ could be replaced with $a$). Note that this means that the first uppercase letter (if it exists) is $D$.
\item To avoid intervals within a region, between every two instances of $a$ (except where the word is $a^2$) there must be at least one $b$. Similarly, between two instances of $b$ there must be either $a$ or $c$, between two $c$s there must be $b$ or $d$, and between two $d$s there must be $c$. The same rules apply for the uppercase letters.
\item There should be no factors of the form $ca$, $da$, $db$, $CA$, $DA$ or $DB$, as these can be encoded as $ac$, $ad$, $bd$ and so on.
\item If it exists, the final $A$ can be followed by at most one letter, since the subsequent points in the regions encoded by $B$, $C$ and $D$ will form an interval. For uniqueness, this final letter will be $B$.
\item If the word does not contain $A$ then the final $a$ can be followed by at most one letter (which for uniqueness will be encoded by $b$). Otherwise, either we can find an interval of size two in the regions encoded by $\{b,c,B,C,D\}$, or there is an instance of $d$ that could equivalently be encoded by $a$.
\end{itemize}
%TODO: Check the above rules are all we need. Also, can we prove these rules suffice?!

These rules together give the following automaton, which we present in two parts, the first encoding the prefix on the letters $\{a,b,c,d\}$, the second encoding the suffix on $\{A,B,C,D\}$. The outward half-edges in the first automaton indicate a permitted transition to the start state of the second automaton.

{\small
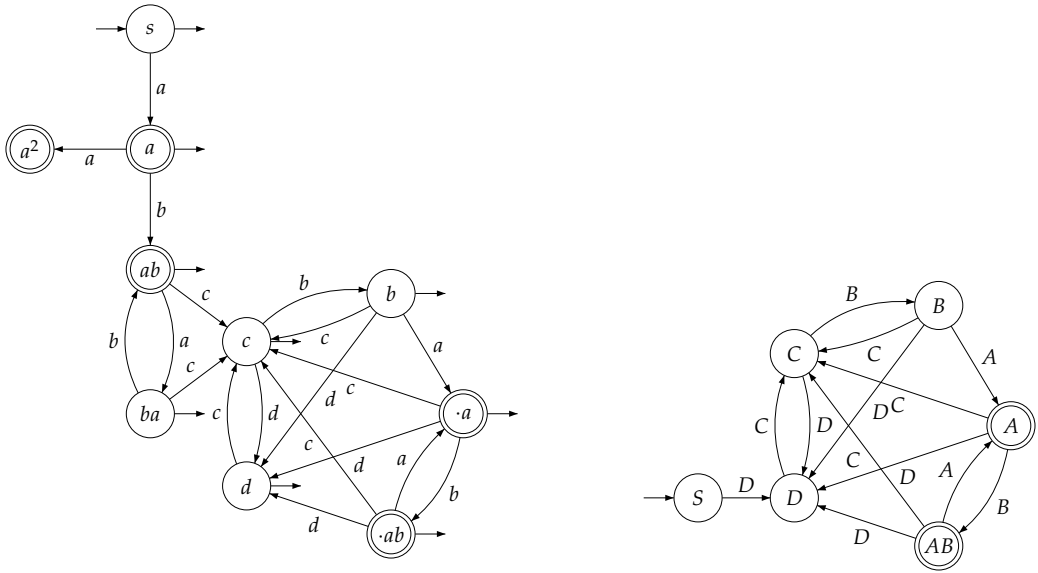
\begin{center}\scalebox{0.8}{
\begin{picture}(90,90)(0,-95)
\node[Nmarks=if](s)(28,-8){$s$}
\node[Nmarks=fr](a)(28,-28){$a$}
\node[Nmarks=r](aa)(8,-28){$a^2$}
\node[Nmarks=fr](ab)(28,-48){$ab$}
\node[Nmarks=f](ba)(28,-72){$ba$}
\node[Nmarks=f](c)(44,-60){$c$}
\node[Nmarks=f](b)(68,-52){$b$}
\node[Nmarks=fr](dota)(80,-72){$\cdot a$}
\node[Nmarks=fr](dotab)(68,-92){$\cdot ab$}
\node[Nmarks=f](d)(44,-84){$d$}
\drawedge(ba,c){$c$}
\drawedge(ab,c){$c$}
\drawedge(a,ab){$b$}
\drawedge(a,aa){$a$}
\drawedge(s,a){$a$}
\drawedge(dotab,d){$d$}
\drawedge(dota,d){$d$}
\drawedge(dotab,c){$c$}
\drawedge(b,dota){$a$}
\drawedge(b,d){$d$}
\drawedge(dota,c){$c$}
\drawqbedge(ab,35.72,-60.06,ba){$a$}
\drawqbedge(ba,19.58,-59.53,ab){$b$}
\drawqbedge(c,52.65,-48.95,b){$b$}
\drawqbedge(b,56.89,-59.27,c){$c$}
\drawqbedge(d,37.57,-71.97,c){$c$}
\drawqbedge(c,49.21,-71.7,d){$d$}
\drawqbedge(dota,78.85,-86.25,dotab){$b$}
\drawqbedge(dotab,69.32,-79.64,dota){$a$}
\end{picture}
\begin{picture}(90,45)(0,-93)
\node[Nmarks=i](S)(28,-84){$S$}
\node(C)(44,-60){$C$}
\node(B)(68,-52){$B$}
\node[Nmarks=r](A)(80,-72){$A$}
\node[Nmarks=r](AB)(68,-92){$AB$}
\node(D)(44,-84){$D$}
\drawedge(S,D){$D$}
\drawedge(AB,D){$D$}
\drawedge(A,D){$D$}
\drawedge(AB,C){$C$}
\drawedge(B,A){$A$}
\drawedge(B,D){$D$}
\drawedge(A,C){$C$}
\drawqbedge(C,52.65,-48.95,B){$B$}
\drawqbedge(B,56.89,-59.27,C){$C$}
\drawqbedge(D,37.57,-71.97,C){$C$}
\drawqbedge(C,49.21,-71.7,D){$D$}
\drawqbedge(A,78.85,-86.25,AB){$B$}
\drawqbedge(AB,69.32,-79.64,A){$A$}
\end{picture}}
\end{center}}

The generating function now follows routinely from the transfer matrix method.
\end{proof}

Of course, by symmetry, the simple permutations of $\F$ are enumerated by the generating function $s_\F(x)$ which is equal to $s_\E(x)$.
%\subsubsection{Counting the overcount}
However the class $\E\cap \F$ also contains simple permutations, which need to be considered in order to handle the overcount in $s_\E(x) + s_\F(x)$.
We describe these in the following result.
%We first make a useful observation:

\begin{lemma}\label{lem-e-cap-f-struct}The simple permutations in $\E\cap\F$ lie in one of the following two grid classes (where a cell with a single dot represents exactly one point):
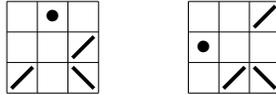
\begin{center}
\begin{tabular}{ccc}
\psset{xunit=0.008in, yunit=0.008in} \psset{linewidth=0.005in}
\begin{pspicture}(0,0)(60,60)
\drawlines(3,3)
\upline(1,1)
\upline(3,2)
\downline(3,1)
\cellclass(2,3){$\bullet$}
\end{pspicture}
&\rule{10pt}{0pt}&
\psset{xunit=0.008in, yunit=0.008in} \psset{linewidth=0.005in}
\begin{pspicture}(0,0)(60,60)
\drawlines(3,3)
\upline(2,1)
\upline(3,3)
\downline(3,1)
\cellclass(1,2){$\bullet$}
\end{pspicture}
\end{tabular}
\end{center}
\end{lemma}

\begin{proof}
Let $\sigma$ be a simple permutation in $\E\cap\F$ of length at least 4.    Suppose  that $\sigma$ had a subsequence $cdab\sim 3412$ and consider a representation of $\sigma$ in the form of the $2\times 2$ class  given in Lemma \ref{lem-two-by-two} (Figure \ref{GridD}).  Then $c$ and $d$ would necessarily lie in the top left cell of the representation.
However, since $\sigma\in\E$, there are no points in the top right cell that separate $c$ and $d$ by value. Similarly, since $\sigma\in\F$, there are no points in the lower left cell that separate $c$ and $d$ by position, and hence  $c$ and $d$ form an interval in $\sigma$, a contradiction. So, $\sigma$ avoids 3412.

%\begin{lemma}\label{lem-e-cap-f-3412}The simple permutations in $\E\cap \F$ avoid the permutation $3412$.\end{lemma}
%
%\begin{proof}
%Suppose to the contrary that $\sigma$ contains $3412$. Note first that $\sigma$ must be $\D$-griddable. In all such $\D$-griddings of $\sigma$, the $3$ and $4$ of any copy of $3412$ must lie in the top left cell. However, since $\sigma\in\E$, there are no points in the top right cell that separate these two points by value. Similarly, since $\sigma\in\F$, there are no points in the lower left cell that separate these two points by position, and hence the $3$ and $4$ form an interval in $\sigma$, a contradiction.
%\end{proof}
%
%\begin{lemma}\label{lem-e-cap-f-struct}The simple permutations in $\E\cap\F$ lie in one of the following two grid classes:
%\begin{center}
%\begin{tabular}{ccc}
%\psset{xunit=0.008in, yunit=0.008in} \psset{linewidth=0.005in}
%\begin{pspicture}(0,0)(60,60)
%\drawlines(3,3)
%\upline(1,1)
%\upline(3,2)
%\downline(3,1)
%\cellclass(2,3){$\bullet$}
%\end{pspicture}
%&\rule{10pt}{0pt}&
%\psset{xunit=0.008in, yunit=0.008in} \psset{linewidth=0.005in}
%\begin{pspicture}(0,0)(60,60)
%\drawlines(3,3)
%\upline(2,1)
%\upline(3,3)
%\downline(3,1)
%\cellclass(1,2){$\bullet$}
%\end{pspicture}
%\end{tabular}
%\end{center}
%\end{lemma}
%
%\begin{proof}
Next note that, since $\sigma$ is simple, it must contain either $2413$ or $3142$, or both (but in fact it will emerge that here this final case is impossible). Suppose first that $\sigma$ contains $2413$, and choose a copy of $2413$ in $\sigma$ with  $2$ as far right as possible, the $3$ as far left as possible and the $1$ as high as possible.
Now we have the situation depicted in Figure~\ref{fig-sigma-2413}. In this diagram, all non-labelled cells are empty, either because $\sigma$ cannot contain any of the patterns $2143$, $4312$ or $3412$, or because of our choice of points for the $2413$-pattern.
\begin{figure}
\centering
\psset{xunit=0.012in, yunit=0.012in} \psset{linewidth=0.01in}
\begin{pspicture}(0,0)(100,100)
\drawlines(5,5)
\cellclass(5,5){$A$}
\cellclass(5,4){$B$}
\cellclass(3,3){$C$}
\cellclass(1,2){$D$}
\cellclass(1,1){$E$}
\cellclass(4,1){$F$}
\cellclass(5,1){$G$}
\pscircle*(20,40){0.03in}
\pscircle*(40,80){0.03in}
\pscircle*(60,20){0.03in}
\pscircle*(80,60){0.03in}
\end{pspicture}
\caption{The case where the simple permutation $\sigma\in \E\cap\F$ contains 2413.}
\label{fig-sigma-2413}
\end{figure}
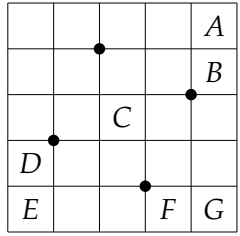

In order to avoid the pattern 2143, the cells labelled $A$ and $B$ together must avoid $21$, i.e.\ they form an increasing sequence. Similarly, cells $D$ and $E$ must avoid $21$ for the same reason, while $F$ and $G$ avoid $12$ in order to ensure that $\sigma$ does not contain $4312$. The cell labelled $C$ can contain at most one point, as otherwise it will form a nontrivial interval. We now claim that $A$ must in fact be empty. For if it were not, then because $\sigma$ cannot end with its maximum there would have to be an element of $G$ to its right. But then $\sigma$ would either begin or end with its minimum, which in either case is a contradiction. Together, these conditions on the cells demonstrate that $\sigma$ lies in the first of the two grid classes in the statement of the Lemma. The situation where $\sigma$ contains $3142$ is analogous, and gives rise to permutations in the second of these classes.
\end{proof}

With the structure of the simples of $\E\cap\F$ described, we are now able to enumerate them. The following encoding and automaton is used to enumerate the simples containing 2413 (i.e. the ones given by the left-hand diagram in Lemma \ref{lem-e-cap-f-struct}):
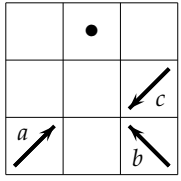
\begin{center}
\begin{tabular}{ccc}
\psset{xunit=0.015in, yunit=0.015in} \psset{linewidth=0.005in}
\begin{pspicture}(0,-15)(60,60)
\drawlines(3,3)
\upline(1,1)
\upline(3,2)
\downline(3,1)
\psline[linestyle=solid,linewidth=0.02in]{->}(3,3)(17,17)
\psline[linestyle=solid,linewidth=0.02in]{<-}(43,23)(57,37)
\psline[linestyle=solid,linewidth=0.02in]{<-}(43,17)(57,3)
\rput(6,14){\footnotesize $a$}
\rput(46,6){\footnotesize $b$}
\rput(54,26){\footnotesize $c$}
\cellclass(2,3){$\bullet$}
\end{pspicture}

&\rule{10pt}{0pt}&
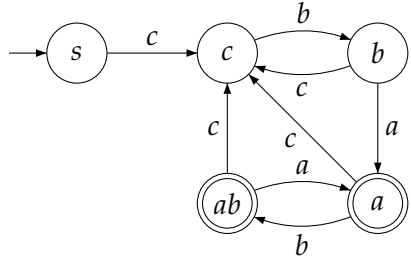
\begin{picture}(60,40)(0,-40)
\node[Nmarks=i](s)(16.0,-12.0){$s$}
\node(c)(36.0,-12.0){$c$}
\node(b)(56.0,-12.0){$b$}
\node[Nmarks=r](a)(56.0,-32.0){$a$}
\node[Nmarks=r](ab)(36.0,-32.0){$ab$}
\drawedge(s,c){$c$}
\drawedge(b,a){$a$}
\drawedge(a,c){$c$}
\drawedge(ab,c){$c$}
\drawqbedge(c,46.3,-6.09,b){$b$}
\drawqbedge(b,45.77,-17.73,c){$c$}
\drawqbedge(a,45.77,-38.1,ab){$b$}
\drawqbedge(ab,46.04,-26.19,a){$a$}
\end{picture}
\end{tabular}
\end{center}

The generating function from this automaton is $\displaystyle \frac{x^3}{1-x-x^2}$, which is the generating function for the Fibonacci numbers, offset by three. Consequently, noting that the sets of simples in the two diagrams of Lemma~\ref{lem-e-cap-f-struct} are disjoint, the generating function for the simples in $\E\cap\F$ is:
\[ s_{\E\cap\F}(x) = \frac{2x^4}{1-x-x^2}\]
Note the extra factor of $x$ which has been added, to account for the singleton cell.

Finally, the generating function for the simples of $\C = \Av(2143,4312)$ of length at least 4 is:

\[ s_\C(x) = s_\E(x)+s_\F(x) - s_{\E\cap\F}(x) = \frac{2x^4(1-2x+x^4)}{(1+x)(1-2x)^2(1-x-x^2)}.\]

%The sequence begins 2, 4, 12, 26, 62, 136, 302, 654, 1412.

\subsection{Inflations of simples}

Our next task is to describe the inflations of simple permutations in the grid class $\E$. We identify two types of simple: the simple permutations whose encoding ends with the letter $A$ will be called \emph{type 1}, and all other simples (i.e.\ those ending with $a$, $b$ or $B$) are \emph{type 2}.

\begin{lemma}
In the grid class $\E$, Type 1 simple permutations are enumerated by \[s^1_\E(x) = \frac{x^4(1-x-x^2)}{(1+x)(1-2x)(1-x-2x^2)},\]
and each point can only be inflated by $\Av(12)$ or $\Av(21)$.
\end{lemma}

\begin{proof}
The generating function $s^1_\E(x)$ follows by considering the automaton for the simple permutations in $\E$, but with the state labelled $A$ as the only accept state.

Second, it is routinely verified that all points can only be inflated by an increasing or decreasing sequence, in order to remain in the grid class $\E$, noting that the regions encoded by $A$, $B$, $C$ and $D$ are all non-empty, and the last point encoded by $A$ lies to the right of all points encoded by $B$. %(How much more needs to be said?)
\end{proof}

\begin{lemma}
In the grid class $\E$, Type 2 simple permutations are enumerated by \[s^2_\E(x) = \frac{x^4(1-3x^2-x^3)}{(1+x)(1-2x)(1-x-2x^2)},\]
and each point can only be inflated by $\Av(12)$ or $\Av(21)$, except for one point (either encoded by $b$ or $B$ if this is the last letter, or encoded by $c$ if $a$ is the final letter), which can be inflated by $\Av(312,2143)$.
\end{lemma}

\begin{proof}
The generating function is clear from the previous lemma. Second, as in the analysis for the Type 1 simples, all points can only be inflated by $\Av(21)$ or $\Av(12)$ with one exception. This exception is as described in the statement of the lemma: note that in all of these cases, there exists some point encoded by $a$ or $A$ that lies above and to the left of the point in question, so any inflation must avoid $312$ to avoid creating a $4312$-pattern, in addition to avoiding $2143$.
\end{proof}

We are now in a position to enumerate the strong-indecomposables in $\E$, but this is not quite what we want: in order to enumerate the strong-indecomposables in $\Av(4312,2143)$, we must also be able to handle the inflations of permutations in $\E\cap\F$. Again, these divide into two types, which are analogous to the types given above: \emph{type 1} corresponds to those simple permutations whose encoding ends with $a$, and \emph{type 2} are those whose encoding ends with $b$. Applying a similar analysis to the above shows that each point in type 1 simples in $\E\cap\F$ can be inflated by $\Av(12)$ or $\Av(21)$ only, and the generating function is:
\[ s^1_{\E\cap\F}(x) = \frac{2x^4}{1-2x^2-x^3}.\]

Similarly, type 2 simples in  $\E\cap\F$ can only be inflated by $\Av(12)$ or $\Av(21)$, with the exception of one point which can be inflated by $\Av(312,2143)$, and the generating function for these simples is:
\[s^2_{\E\cap\F}(x) = \frac{2x^5}{1-2x^2-x^3}.\]

The generating function for the permutations in $\Av(312,2143)$ is given by \[g(x) = \frac{x(1-x)}{1-3x+x^2},\]
from which it follows that the strong-indecomposables in $\Av(4312,2143)$ of length at least $4$ have generating function:

\begin{eqnarray*}
f_{ind}(x) &=& \left[2s^1_\E\left(\frac{x}{1-x}\right)-s^1_{\E\cap\F}\left(\frac{x}{1-x}\right)\right] + g(x)\frac{1-x}{x}\left[ 2s^2_\E\left(\frac{x}{1-x}\right)-s^2_{\E\cap\F}\left(\frac{x}{1-x}\right)\right]\\
&=&\frac{2x^4(1-7x+17x^2-18x^3+11x^4-5x^5)}{(1-x)^2(1-3x)^2(1-3x+x^2)^2}.
\end{eqnarray*}

\subsection{Completing the enumeration}

To complete the enumeration, we first need to count the sum- and skew-decomposable permutations in $\Av(4312,2143)$, given by generating functions $f_\oplus(x)$ and $f_\ominus(x)$ respectively. The generating function of  the entire class will be denoted by $f(x)$.

Consider a sum-decomposable permutation $\pi=\pi_1\oplus\pi_2$, where $\pi_1$ is sum-indecomposable. If $\pi_1=1$, then $\pi_2$ can be any permutation from $\Av(4312,2143)$, while if $\pi_1$ has more than one point then it contains $21$ (since it is sum-indecomposable, and hence not increasing), and so $\pi_2\in \Av(21)$ in order to avoid creating a $2143$-pattern. In this latter case, $\pi_1$ must either be skew-decomposable, or strong-indecomposable of length at least $4$. Hence:
\[ f_\oplus(x) = xf(x) + (f_{ind}(x)+f_{\ominus}(x))\frac{x}{1-x}.\]

For the skew-decomposables, write $\pi=\pi_1\ominus\pi_2$ where we take $\pi_2$ to be skew-indecomposable. Note that, since $\pi_1$ is necessarily nonempty, we have $\pi_2\in\Av(312,2143)$ to avoid a $4312$-pattern. If $\pi_2=1$, then $\pi_1$ can be any permutation from $\Av(4312,2143)$, while if $\pi_2$ is a non-trivial skew-indecomposable in $\Av(312,2143)$, then $\pi_1\in\Av(21)$ in order to avoid creating a $4312$-pattern.

The generating function for the skew-indecomposables of length at least $2$ in $\Av(312,2143)$ is (see, e.g.~\cite{albert:the-enumeration:2143:4231})
\[g_{\not\ominus}(x) = \frac{x(1-x)^2}{1-3x+x^2}-x=\frac{x^2}{1-3x+x^2},\]
from which it follows that
\[f_\ominus(x) = xf(x) + \frac{x^2}{1-3x+x^2}\cdot \frac{x}{1-x}.\]

Substituting this expression for $f_\ominus(x)$ into the expression for $f_\oplus(x)$, and noting that $f(x)=1+x+f_\oplus(x)+f_\ominus(x)+f_{ind}(x)$ allows us to solve for $f(x)$:

\begin{eqnarray*}
f(x) &=& \frac{1-x}{1-3x+x^2}\cdot\left(1+x+\frac{f_{ind}(x)}{1-x}+\frac{x^3}{(1-3x+x^2)(1-x)^3}\right)\\
&=&\frac{1-13x+69x^2-191x^3+294x^4-252x^5+116x^6-23x^7}{(1-x)^2(1-3x)^2(1-3x+x^2)^2}
\end{eqnarray*}

The series expansion begins 1, 2, 6, 22, 86, 337, 1295, 4854, 17760, 63594, 223488, 772841, \ldots  (sequence A165529 of~\cite{sloane:the-on-line-enc:}).

\section{\texorpdfstring{$\Av(1324, 4312)$}{Av(1324, 4312)}}

\subsection{Simple permutations}

The main result of this subsection is that every simple permutation of $\Av(1324, 4312)$ has one of the following grid class forms:

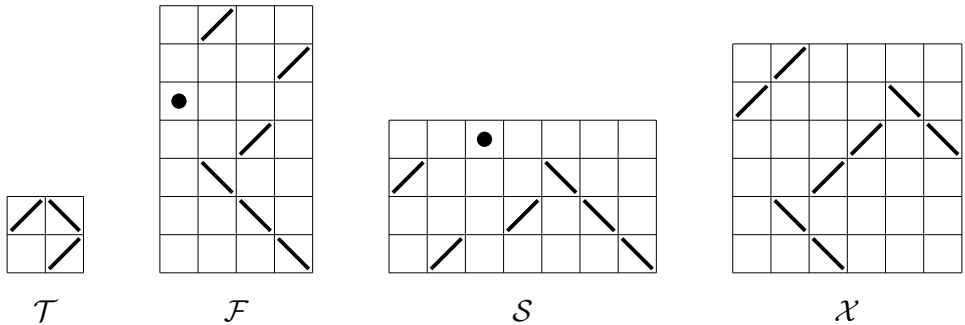
\begin{center}
\psset{xunit=0.01in, yunit=0.01in} \psset{linewidth=0.006in}
\begin{pspicture}(0,0)(500,160)
%T
\psline[linestyle=solid,linewidth=0.005in](0,20)(40,20)
\psline[linestyle=solid,linewidth=0.005in](0,40)(40,40)
\psline[linestyle=solid,linewidth=0.005in](0,60)(40,60)
\psline[linestyle=solid,linewidth=0.005in](0,20)(0,60)
\psline[linestyle=solid,linewidth=0.005in](20,20)(20,60)
\psline[linestyle=solid,linewidth=0.005in](40,20)(40,60)
\psline[linestyle=solid,linewidth=0.02in](2,42)(18,58)
\psline[linestyle=solid,linewidth=0.02in](22,58)(38,42)
\psline[linestyle=solid,linewidth=0.02in](22,22)(38,38)
\rput(20,0){$\T$}

%F
\psline[linestyle=solid,linewidth=0.005in](80,20)(160,20)
\psline[linestyle=solid,linewidth=0.005in](80,40)(160,40)
\psline[linestyle=solid,linewidth=0.005in](80,60)(160,60)
\psline[linestyle=solid,linewidth=0.005in](80,80)(160,80)
\psline[linestyle=solid,linewidth=0.005in](80,100)(160,100)
\psline[linestyle=solid,linewidth=0.005in](80,120)(160,120)
\psline[linestyle=solid,linewidth=0.005in](80,140)(160,140)
\psline[linestyle=solid,linewidth=0.005in](80,160)(160,160)
\psline[linestyle=solid,linewidth=0.005in](80,20)(80,160)
\psline[linestyle=solid,linewidth=0.005in](100,20)(100,160)
\psline[linestyle=solid,linewidth=0.005in](120,20)(120,160)
\psline[linestyle=solid,linewidth=0.005in](140,20)(140,160)
\psline[linestyle=solid,linewidth=0.005in](160,20)(160,160)
\psline[linestyle=solid,linewidth=0.02in](102,142)(118,158)
\psline[linestyle=solid,linewidth=0.02in](142,122)(158,138)
\pscircle*(90,110){0.04in}
\psline[linestyle=solid,linewidth=0.02in](122,82)(138,98)
\psline[linestyle=solid,linewidth=0.02in](102,78)(118,62)
\psline[linestyle=solid,linewidth=0.02in](122,58)(138,42)
\psline[linestyle=solid,linewidth=0.02in](142,38)(158,22)
\rput(120,0){$\F$}
%S
\psline[linestyle=solid,linewidth=0.005in](200,20)(340,20)
\psline[linestyle=solid,linewidth=0.005in](200,40)(340,40)
\psline[linestyle=solid,linewidth=0.005in](200,60)(340,60)
\psline[linestyle=solid,linewidth=0.005in](200,80)(340,80)
\psline[linestyle=solid,linewidth=0.005in](200,100)(340,100)
\psline[linestyle=solid,linewidth=0.005in](200,20)(200,100)
\psline[linestyle=solid,linewidth=0.005in](220,20)(220,100)
\psline[linestyle=solid,linewidth=0.005in](240,20)(240,100)
\psline[linestyle=solid,linewidth=0.005in](260,20)(260,100)
\psline[linestyle=solid,linewidth=0.005in](280,20)(280,100)
\psline[linestyle=solid,linewidth=0.005in](300,20)(300,100)
\psline[linestyle=solid,linewidth=0.005in](320,20)(320,100)
\psline[linestyle=solid,linewidth=0.005in](340,20)(340,100)
\pscircle*(250,90){0.04in}
\psline[linestyle=solid,linewidth=0.02in](202,62)(218,78)
\psline[linestyle=solid,linewidth=0.02in](222,22)(238,38)
\psline[linestyle=solid,linewidth=0.02in](262,42)(278,58)
\psline[linestyle=solid,linewidth=0.02in](282,78)(298,62)
\psline[linestyle=solid,linewidth=0.02in](302,58)(318,42)
\psline[linestyle=solid,linewidth=0.02in](322,38)(338,22)
\rput(270,0){$\S$}
%X
\psline[linestyle=solid,linewidth=0.005in](380,20)(500,20)
\psline[linestyle=solid,linewidth=0.005in](380,40)(500,40)
\psline[linestyle=solid,linewidth=0.005in](380,60)(500,60)
\psline[linestyle=solid,linewidth=0.005in](380,80)(500,80)
\psline[linestyle=solid,linewidth=0.005in](380,100)(500,100)
\psline[linestyle=solid,linewidth=0.005in](380,120)(500,120)
\psline[linestyle=solid,linewidth=0.005in](380,140)(500,140)
\psline[linestyle=solid,linewidth=0.005in](380,20)(380,140)
\psline[linestyle=solid,linewidth=0.005in](400,20)(400,140)
\psline[linestyle=solid,linewidth=0.005in](420,20)(420,140)
\psline[linestyle=solid,linewidth=0.005in](440,20)(440,140)
\psline[linestyle=solid,linewidth=0.005in](460,20)(460,140)
\psline[linestyle=solid,linewidth=0.005in](480,20)(480,140)
\psline[linestyle=solid,linewidth=0.005in](500,20)(500,140)
\psline[linestyle=solid,linewidth=0.02in](382,102)(398,118)
\psline[linestyle=solid,linewidth=0.02in](402,122)(418,138)
\psline[linestyle=solid,linewidth=0.02in](402,58)(418,42)
\psline[linestyle=solid,linewidth=0.02in](422,62)(438,78)
\psline[linestyle=solid,linewidth=0.02in](422,38)(438,22)
\psline[linestyle=solid,linewidth=0.02in](442,82)(458,98)
\psline[linestyle=solid,linewidth=0.02in](462,118)(478,102)
\psline[linestyle=solid,linewidth=0.02in](482,98)(498,82)
\rput(440,0){$\X$}
\end{pspicture}
\end{center}

The names of the classes are chosen mnemonically since $\T$ is a \underline{T}wo-by four class, $\F$ is \underline{F}our-by-seven, $\S$ is \underline{S}even-by-four and $\X$ is si\underline{X}-by-six.  In the diagrams for $\F$ and $\S$ the cells containing a dot represent a single point or no point at all.  Cells containing monotone sequences may be empty (although simplicity places some constraints on this).

Notice that all these classes are certainly contained in $\Av(1324, 4312)$.  We shall now consider an arbitrary simple permutation $\pi\in\Av(1324, 4312)$ and prove it has one of the above forms.

%We now know that every simple permutation lies in one of 4 classes.  Michael's proof of this had the following steps:
%
%\begin{enumerate}
%\item the simple permutations of $\Av(4312, 1324, 41352, 31524, 25314)$ lie in the $2\times 2$ class
%\item a diagram chase inside $\Av(1324, 4312\}$ of simple permutations containing one of 41352, 31524, 25314; the diagram chase proves that such a simple permutation lies in one of a small number of explicitly known monotone grid classes.  It is hoped that these could then be inspected and proved to lie in one of the other 3 purported classes
%\end{enumerate}
%
%In addition I have a more traditional proof.  This works by looking at each of 2143, 3412, 3142, 2413 as the outer boundary and laboriously following all the case trails.  At the moment this is only represented by 10-20 files containing diagrams and text boxes but I have checked it several times and believe it is correct.  I have recently found a common beginning to all of these cases which is now described.

We identify the first point $m$ and the greatest point $n$ of $\pi$ and use them to divide $\pi$ into 4 quadrants.

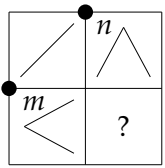
\begin{center}
\psset{xunit=0.010in, yunit=0.010in} \psset{linewidth=0.005in}
\begin{pspicture}(0,0)(80,80)
\psline[linestyle=solid,linewidth=0.005in](0,0)(80,0)
\psline[linestyle=solid,linewidth=0.005in](0,40)(80,40)
\psline[linestyle=solid,linewidth=0.005in](0,80)(80,80)
\psline[linestyle=solid,linewidth=0.005in](0,0)(0,80)
\psline[linestyle=solid,linewidth=0.005in](40,0)(40,80)
\psline[linestyle=solid,linewidth=0.005in](80,0)(80,80)
\pscircle*(0,40){0.04in}
\pscircle*(40,80){0.04in}
\psline[linestyle=solid,linewidth=0.005in](6,46)(34,74)
\psline[linestyle=solid,linewidth=0.005in](46,46)(60,72)
\psline[linestyle=solid,linewidth=0.005in](60,72)(74,46)
\psline[linestyle=solid,linewidth=0.005in](8,20)(34,34)
\psline[linestyle=solid,linewidth=0.005in](8,20)(34,6)
\rput(13,32){$m$}
\rput(50,72){$n$}
\rput(60,20){?}
\end{pspicture}
\end{center}

The NW quadrant is increasing because of 1324 avoidance.  The NE quadrant is wedge-shaped as $\wedge$ because of 312 and 213 avoidance, the SW quadrant is wedge shaped as $<$ because of 312 and 213 avoidance, and as yet the SE quadrant is unknown.

Consider the wedge in the NE quadrant.  We shall prove it is monotone.  If not then there is a proper apex and this point together with the two on either side forms either $acb\sim132$ or $bca\sim 231$.  Consider the first of these cases.  Since $cb$ is not an interval it has a splitting point $p$ which must lie in the NW quadrant or the SE quadrant.  In the former case $mpac\sim 1324$ and in the latter case $ncpb\sim 4312$ both of which are impossible.  Now consider the second case.  Here $bc$ is not an interval so has a splitting point $p$.  If $p$ is in the NW quadrant $mpbc\sim 1324$ and if $p$ is in the SE quadrant $nbpa\sim 4312$; again both of these are impossible.

Similarly (or by symmetry) the SW quadrant is monotone.

These two monotone quadrants cannot both be increasing and of length more than 1.  For if we had $ab\sim 12$ in the SW quadrant and $cd\sim 12$ in the NE quadrant each chosen as closest possible pairs then, as they are not intervals, they must each have separators elsewhere.  The separators cannot lie in the NW quadrant (else a 1324 pattern occurs) so must lie in the SE quadrant.  The $ab$ separator ($x$ say) must lie to the right of $d$ (else 1324 occurs again) and the $cd$ separator ($y$ say) must lie below $a$ (again because of 1324).  But then $ncxy\sim 4312$ which is also impossible.

We therefore have three possibilities: the NE quadrant is increasing of length more than 1 and the SW quadrant is decreasing, the NE quadrant is decreasing and the SW quadrant is increasing of length more than 1, or both of these quadrants are decreasing.  We shall see that these possibilities lead to $\pi\in \F$, $\pi\in \S$, or $\pi\in \T \cup \X$ respectively.

\begin{lemma}
If the NE (respectively SW) quadrant is increasing of length more than 1 then $\pi$ lies in $\F$ (respectively $\S$).
\end{lemma}
\begin{proof}
Let $a,b$ be the last two points in the NE quadrant.  Since $ab$ is not an interval and it has no left separator (by 1324) there is a separator  $s$ in the SE quadrant.  In fact the point 1 is a separator.  To see this we rule out the alternatives:

\begin{itemize}
\item $1$ is to the left of $a$.  Then $1asb\sim 1324$.
\item $1$ is to the right of $b$.  Then no point comes after 1 for any successor $u$ would lie in the SE quadrant and $nb1u\sim 4312$.  But a simple permutation cannot end with its minimum.
\end{itemize}

It also follows that $1$ is succeeded only by $b$ (because another successor $v$ would lie in the SE quadrant and then $na1v\sim 4312$).

Let $\delta$ be the sequence of points in the SW quadrant (known to form a decreasing sequence). In the SE quadrant  let $P$, $Q$, $R$ be the sets of points entirely  below $\delta$, between the first and last points of $\delta$, and entirely above $\delta$ respectively (all of which come before the point 1).  We have

\begin{itemize}
\item $P$ is decreasing (or there is a 4312 pattern with the initial point $m$ and a point of $\delta$ as the 4 and 3),
\item $Q$ is empty (or again 4312 with the initial point $m$ as the 4, and the top and bottom points of $\delta$ as the 3 and 1),
\item $R$ is increasing and, with the NE quadrant, is a single increasing sequence (because of 1324 avoidance and the fact that these points come before $b$).
\end{itemize}

The situation is now as shown, where we have also taken into account the fact that no point in the NW quadrant can separate any pair in the NE quadrant.
%
%\begin{center}
%\psset{xunit=0.03in, yunit=0.03in} \psset{linewidth=0.005in}
%\begin{pspicture}(0,0)(80,80)
%\psline[linestyle=solid,linewidth=0.005in](0,0)(80,0)
%\psline[linestyle=solid,linewidth=0.005in](0,40)(80,40)
%\psline[linestyle=solid,linewidth=0.005in](0,80)(80,80)
%\psline[linestyle=solid,linewidth=0.005in](0,0)(0,80)
%\psline[linestyle=solid,linewidth=0.005in](40,0)(40,80)
%\psline[linestyle=solid,linewidth=0.005in](80,0)(80,80)
%\psline[linestyle=solid,linewidth=0.02in](6,46)(36,76)
%\pscircle*(0,40){0.04in}
%\pscircle*(40,80){0.04in}
%\rput(5,43){$m$}
%\rput(45,77){$n$}
%\psline[linestyle=solid,linewidth=0.02in](13,26)(26,13)
%\pscircle*(13,26){0.04in}
%\pscircle*(26,13){0.04in}
%\pscircle*(80,67){0.04in}
%\rput(76,67){$b$}
%\pscircle*(76,63){0.04in}
%\rput(72,63){$a$}
%\psline[linestyle=dashed,linewidth=0.005in](78,63)(78,0)
%\pscircle*(78,0){0.04in}
%\pscircle*(68,55){0.04in}
%\psline[linestyle=dashed,linewidth=0.005in](76,63)(68,55)
%\psline[linestyle=dashed,linewidth=0.005in](0,13)(80,13)
%\psline[linestyle=dashed,linewidth=0.005in](0,26)(80,26)
%\psline[linestyle=solid,linewidth=0.02in](43,10)(76,3)
%\psline[linestyle=solid,linewidth=0.02in](43,30)(65,37)
%
%\end{pspicture}
%\end{center}

%Next we recall that no point in the NW quadrant can separate any pair in the NE quadrant.  So we have

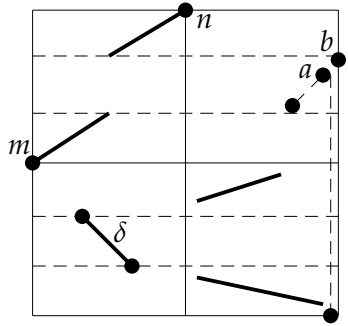
\begin{center}
\psset{xunit=0.02in, yunit=0.02in} \psset{linewidth=0.005in}
\begin{pspicture}(0,0)(80,80)
\psline[linestyle=solid,linewidth=0.005in](0,0)(80,0)
\psline[linestyle=solid,linewidth=0.005in](0,40)(80,40)
\psline[linestyle=solid,linewidth=0.005in](0,80)(80,80)
\psline[linestyle=solid,linewidth=0.005in](0,0)(0,80)
\psline[linestyle=solid,linewidth=0.005in](40,0)(40,80)
\psline[linestyle=solid,linewidth=0.005in](80,0)(80,80)
\pscircle*(0,40){0.04in}
\pscircle*(40,80){0.04in}
\rput(-3.5,44){$m$}
\rput(45,77){$n$}
\psline[linestyle=solid,linewidth=0.02in](13,26)(26,13)
\pscircle*(13,26){0.04in}
\pscircle*(26,13){0.04in}
\pscircle*(80,67){0.04in}
\pscircle*(76,63){0.04in}
\rput(23,22){$\delta$}
\rput(77,72){$b$}
\rput(71.5,64){$a$}
\psline[linestyle=dashed,linewidth=0.005in](78,63)(78,0)
\pscircle*(78,0){0.04in}
\pscircle*(68,55){0.04in}
\psline[linestyle=dashed,linewidth=0.005in](76,63)(68,55)
\psline[linestyle=dashed,linewidth=0.005in](0,13)(80,13)
\psline[linestyle=dashed,linewidth=0.005in](0,26)(80,26)
\psline[linestyle=solid,linewidth=0.02in](43,10)(76,3)
\psline[linestyle=solid,linewidth=0.02in](43,30)(65,37)
\psline[linestyle=dashed,linewidth=0.005in](0,53)(80,53)
\psline[linestyle=dashed,linewidth=0.005in](0,68)(80,68)
\psline[linestyle=solid,linewidth=0.02in](0,40)(20,53)
\psline[linestyle=solid,linewidth=0.02in](20,68)(40,80)

\end{pspicture}
\end{center}
Finally, the increasing sequence shown that begins at  $m$ must consist of $m$ alone.  For let $p$ be the next point in this sequence.  As $mp$ is not an interval the first point of $\delta$ must separate these two points.  But then either $dprb$ with $r\in R$ is a 1324 pattern or $R$ is empty and $md$ is then an interval.
%lie entirely before the decreasing segment in the SW quadrant because of 1324 avoidance (unless the set $R$ is empty there is a degeneracy which is easily seen to come to the $4\times 2$ subclass) and then we have class $F$.

This gives the class $\F$ and the other case of the lemma follows by symmetry.
\end{proof}

\begin{lemma}
If the NE and SW quadrants are decreasing then the permutation lies in $\T$ or $\X$.
% (or a subclass).
\end{lemma}
\begin{proof}
We begin by considering the effect of 1324 and 4312 avoidance on the SE quadrant:
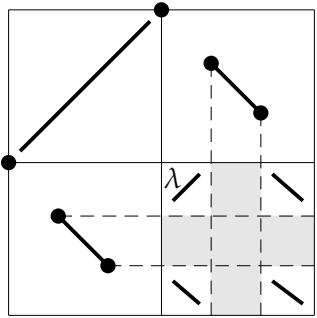
\begin{center}
\psset{xunit=0.02in, yunit=0.02in} \psset{linewidth=0.005in}
\begin{pspicture}(0,0)(80,80)
\pspolygon[fillstyle=solid,linecolor=lightgray,fillcolor=lightgray](40,13)(80,13)(80,26)(40,26)
\pspolygon[fillstyle=solid,linecolor=lightgray,fillcolor=lightgray](53,0)(66,0)(66,40)(53,40)
\psline[linestyle=solid,linewidth=0.005in](0,0)(80,0)
\psline[linestyle=solid,linewidth=0.005in](0,40)(80,40)
\psline[linestyle=solid,linewidth=0.005in](0,80)(80,80)
\psline[linestyle=solid,linewidth=0.005in](0,0)(0,80)
\psline[linestyle=solid,linewidth=0.005in](40,0)(40,80)
\psline[linestyle=solid,linewidth=0.005in](80,0)(80,80)
\pscircle*(0,40){0.04in}
\pscircle*(40,80){0.04in}
%\rput(15,50){m}
%\rput(55,70){n}

\psline[linestyle=solid,linewidth=0.02in](13,26)(26,13)
\pscircle*(13,26){0.04in}
\pscircle*(26,13){0.04in}

\psline[linestyle=solid,linewidth=0.02in](53,66)(66,53)
\pscircle*(53,66){0.04in}
\pscircle*(66,53){0.04in}

\psline[linestyle=dashed,linewidth=0.005in](13,26)(80,26)
\psline[linestyle=dashed,linewidth=0.005in](26,13)(80,13)
\psline[linestyle=dashed,linewidth=0.005in](53,66)(53,00)
\psline[linestyle=dashed,linewidth=0.005in](66,53)(66,0)

%\rput(47,20){0}
%\rput(60,20){0}
%\rput(73,20){0}
%\rput(60,33){0}
%\rput(60,7){0}

\psline[linestyle=solid,linewidth=0.02in](3,43)(37,77)
\psline[linestyle=solid,linewidth=0.02in](43,30)(50,37)

\psline[linestyle=solid,linewidth=0.02in](69,37)(77,30)
\psline[linestyle=solid,linewidth=0.02in](43,9)(50,3)
\psline[linestyle=solid,linewidth=0.02in](69,9)(77,3)

\rput(43,36){$\lambda$}

\end{pspicture}
\end{center}

In this diagram the SE quadrant has been partitioned into 9 cells.  The cell with content $\lambda$ is increasing because of 1324 avoidance, the 5 gray cells are all empty because of  4312 avoidance, and the remaining 3  cells are decreasing also because of 4312 avoidance.

Now consider the decreasing content in the bottom right corner of the SE quadrant.  To avoid 4312, any point in this cell comes after and below both the other two decreasing cells in this quadrant. Thus it is a skew component and therefore empty.

If the cell marked $\lambda$ is empty we have class $\T$.  If not the situation in the NW cell is as shown below with no point in the region marked ``Empty'' otherwise we would have a 1324 pattern with a point of $\lambda$ playing the role of 2.

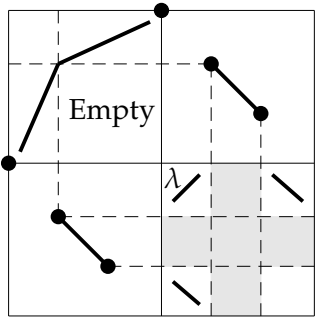
\begin{center}
\psset{xunit=0.02in, yunit=0.02in} \psset{linewidth=0.005in}
\begin{pspicture}(0,0)(80,80)
\pspolygon[fillstyle=solid,linecolor=lightgray,fillcolor=lightgray](40,13)(80,13)(80,26)(40,26)
\pspolygon[fillstyle=solid,linecolor=lightgray,fillcolor=lightgray](53,0)(66,0)(66,40)(53,40)

\psline[linestyle=solid,linewidth=0.005in](0,0)(80,0)
\psline[linestyle=solid,linewidth=0.005in](0,40)(80,40)
\psline[linestyle=solid,linewidth=0.005in](0,80)(80,80)
\psline[linestyle=solid,linewidth=0.005in](0,0)(0,80)
\psline[linestyle=solid,linewidth=0.005in](40,0)(40,80)
\psline[linestyle=solid,linewidth=0.005in](80,0)(80,80)
\pscircle*(0,40){0.04in}
\pscircle*(40,80){0.04in}
%\rput(15,50){m}
%\rput(55,70){n}

\psline[linestyle=solid,linewidth=0.02in](13,26)(26,13)
\pscircle*(13,26){0.04in}
\pscircle*(26,13){0.04in}

\psline[linestyle=solid,linewidth=0.02in](53,66)(66,53)
\pscircle*(53,66){0.04in}
\pscircle*(66,53){0.04in}

\psline[linestyle=dashed,linewidth=0.005in](13,26)(80,26)
\psline[linestyle=dashed,linewidth=0.005in](26,13)(80,13)
\psline[linestyle=dashed,linewidth=0.005in](53,66)(53,00)
\psline[linestyle=dashed,linewidth=0.005in](66,53)(66,0)

%\rput(47,20){0}
%\rput(60,20){0}
%\rput(73,20){0}
%\rput(60,33){0}
%\rput(60,7){0}

% \psline[linestyle=solid,linewidth=0.02in](3,43)(37,77)
\psline[linestyle=solid,linewidth=0.02in](43,30)(50,37)

\psline[linestyle=solid,linewidth=0.02in](69,37)(77,30)
\psline[linestyle=solid,linewidth=0.02in](43,9)(50,3)
% \psline[linestyle=solid,linewidth=0.02in](69,9)(77,3)

\rput(43,36){$\lambda$}

\rput(27,53){Empty}

\psline[linestyle=dashed,linewidth=0.005in](13,26)(13,80)
\psline[linestyle=dashed,linewidth=0.005in](53,66)(0,66)

\psline[linestyle=solid,linewidth=0.02in](3,43)(13,66)
\psline[linestyle=solid,linewidth=0.02in](13,66)(37,77)

\end{pspicture}
\end{center}
Finally, to avoid 4312, any point of $\lambda$ cannot simultaneously lie to the left of a point in the decrease below it, and above a point in the decrease to its right. This condition yields the class $\X$.
\end{proof}

\subsection{Introduction to the enumeration}

The automata in this section are the most complex in the paper, yet their construction is not, in principle, very difficult.  So, rather than set out every step in detail, we shall instead trust the reader to supply the various rules for making encodings unique in the same way we did in the previous section.  A guiding principle in selecting the encoding word for a permutation is to use the lexicographically least of all the possibilities.
%The general rule we use when constructing the automata for simple permutations, we label greedily in alphabetical order. Thus, among all possible words that could encode a given simple permutation, we use the one that is lexicographically smallest. I think.
%First, a few things we need: $s_C=$ generating function for the simples in $C$, $f_C=$ g.f. for the whole class, $i_C=$g.f. for the inflations of simples in $C$.

We have to consider a number of sets $C$ of permutations and, uniformly, we let $f_C$ denote the generating function for $C$, $s_C$ the generating function for the simples (of length more than 2) in $C$, and $i_C$ the generating function for the inflations in $\Av(1324, 4312)$ of these simples (these inflations are, of course, the strong-indecomposables).

Apart from the classes $\T,\S,\F,\X$ a number of other sets of permutations feature in our analysis:

\begin{itemize}
\item
$\I = \Av(21)$ and $\D = Av(12)$
\item $\E=\Av(213,312)$, which we previously denoted by $\wedge$, (and its symmetries) whose generating function is $f_\E(x) = \frac{x}{1-2x}$.
\item $\G=\Av(213,4312)$  (and its symmetries), with generating function $f_\G(x) = \frac{x(1-3x+3x^2)}{(1-x)(1-2x)^2}$ (see, e.g.~\cite{atkinson:restricted-perm:}).
\item $W=$ `Wedge simples': Wedge permutations of length more than 2 are never simple but there are simple permutations which are almost wedges as shown here ($<\hspace{-5pt}\cdot$ for brevity).  There are two of this $<\hspace{-5pt}\cdot$ type in every length from length 4 onwards.
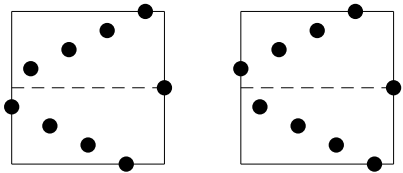
\begin{center}
\psset{xunit=0.01in, yunit=0.01in} \psset{linewidth=0.005in}
\begin{pspicture}(0,0)(200,80)
\psline[linestyle=solid,linewidth=0.005in](0,0)(0,80)
\psline[linestyle=solid,linewidth=0.005in](0,0)(80,0)
\psline[linestyle=solid,linewidth=0.005in](80,0)(80,80)
\psline[linestyle=solid,linewidth=0.005in](0,80)(80,80)
\psline[linestyle=solid,linewidth=0.005in](120,0)(120,80)
\psline[linestyle=solid,linewidth=0.005in](120,0)(200,0)
\psline[linestyle=solid,linewidth=0.005in](200,0)(200,80)
\psline[linestyle=solid,linewidth=0.005in](120,80)(200,80)
\psline[linestyle=dashed,linewidth=0.005in](0,40)(80,40)

\pscircle*(80,40){0.04in}
\pscircle*(0,30){0.04in}
\pscircle*(20,20){0.04in}
\pscircle*(40,10){0.04in}
\pscircle*(60,0){0.04in}
\pscircle*(10,50){0.04in}
\pscircle*(30,60){0.04in}
\pscircle*(50,70){0.04in}
\pscircle*(70,80){0.04in}

\pscircle*(200,40){0.04in}
\pscircle*(120,50){0.04in}
\pscircle*(130,30){0.04in}
\pscircle*(150,20){0.04in}
\pscircle*(170,10){0.04in}
\pscircle*(190,0){0.04in}
\pscircle*(140,60){0.04in}
\pscircle*(160,70){0.04in}
\pscircle*(180,80){0.04in}
\psline[linestyle=dashed,linewidth=0.005in](120,40)(200,40)
\end{pspicture}
\end{center}
Symmetries of this set  under reverse-inverse-reverse symmetry (with shape $\wedge\hspace{-5pt}_{.}\ $) will also be called wedge simples.  These two varieties of wedge simples intersect in $\{2413, 3142, 24153\}$.

\item Sets `$4\times 7$' and `$2\times 4$': we partition $\F$ into two subsets. If the singleton cell of a (simple) permutation in class $\F$ is empty, the permutation reduces to the form:
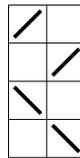
\begin{center}
\psset{xunit=0.01in, yunit=0.01in} \psset{linewidth=0.005in}
\begin{pspicture}(0,0)(40,80)
\drawlines(2,4)
\psline[linestyle=solid,linewidth=0.02in](3,63)(17,77)
\psline[linestyle=solid,linewidth=0.02in](3,37)(17,23)
\psline[linestyle=solid,linewidth=0.02in](23,43)(37,57)
\psline[linestyle=solid,linewidth=0.02in](23,17)(37,3)
\end{pspicture}
\end{center}
We call this class of permutations $2\times 4$ whereas the set of permutations in which the singleton cell does contain a point is called $4\times 7$ (but note that some permutations of $4\times 7$ may be represented in a smaller grid than 4 columns and 7 rows).
\item Sets `$3\times 3$' and `$6\times 6$': in the class $\X$ there is a subset of permutations of the form
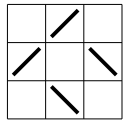
\begin{center}
\psset{xunit=0.01in, yunit=0.01in} \psset{linewidth=0.005in}
\begin{pspicture}(0,0)(60,60)
\drawlines(3,3)
\psline[linestyle=solid,linewidth=0.02in](3,23)(17,37)
\psline[linestyle=solid,linewidth=0.02in](23,43)(37,57)
\psline[linestyle=solid,linewidth=0.02in](23,17)(37,3)
\psline[linestyle=solid,linewidth=0.02in](57,23)(43,37)
\end{pspicture}
\end{center}
We call this set of permutations $3\times 3$.  Permutations of $\X$ not in $3\times 3$ form the set $6\times 6$ (and again note that some of these may be represented in fewer than 6 rows and columns).
\end{itemize}
%Also, let's use $W$ to refer to the (inflations of) wedge simple permutations (2 wedge simples of each length from length 4 -- for the record, these are type 1 wedge simples, so they look like $<\hspace{-5pt}\cdot$).

%We also want two other classes:
%
%$E=\Av(213,312)$ (and its symmetries), whose generating function is $f_E(x) = \frac{x}{1-2x}$.
%
%$G=\Av(213,4312)$ (and its symmetries), with generating function $f_G(x) = \frac{x(1-3x+3x^2)}{(1-x)(1-2x)^2}$ (see, e.g.~\cite{atkinson:restricted-perm:}).

%Incidentally, Murphy~\cite{murphy:restricted-perm:} gives us (without proof) a grid class representation of $G$:
%\begin{center}
%\psset{xunit=0.008in, yunit=0.008in} \psset{linewidth=0.005in}
%\begin{pspicture}(0,0)(80,80)
%\drawlines(5,3)
%\downline(4,2)
%\downline(5,1)
%\upline(1,1)
%\upline(2,3)
%\upline(3,2)
%\end{pspicture}
%\end{center}
%
%The proof for this structure is (I think) quite easy: start with a permutation in $\Av(213,312)$ (looks like $\Lambda$), work from bottom to top, and consider the first point that introduces a copy of $312$. We also want a symmetry of the class $G$, namely $\Av(132,4312)$.

%
%
%
%
%
%
%
%
%
%
%

%%%%%%%%%%%%%%%%%%%%%%%%%%%%%%%%%%%%%%%%%%%
\subsection{The two by two class \texorpdfstring{$\T$}{T}}

We enumerate the simples of length $\geq 4$ using the following encoding scheme and automaton:

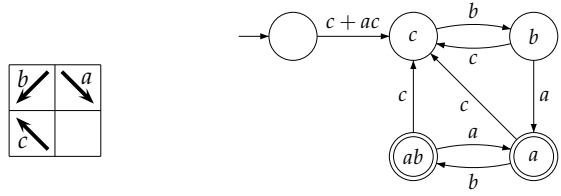
\begin{center}
\begin{tabular}{ccc}
\psset{xunit=0.012in, yunit=0.012in} \psset{linewidth=0.005in}
\begin{pspicture}(0,0)(40,40)
\drawlines(2,2)
\psline[linestyle=solid,linewidth=0.02in]{<-}(3,23)(17,37)
\psline[linestyle=solid,linewidth=0.02in]{<-}(3,17)(17,3)
\psline[linestyle=solid,linewidth=0.02in]{->}(23,37)(37,23)
\rput(34,34){\footnotesize $a$}
\rput(6,34){\footnotesize $b$}
\rput(6,6){\footnotesize $c$}
%\cellclass(2,3){$\bullet$}
\end{pspicture}
&\rule{10pt}{0pt}&
{\small
\scalebox{0.8}{
\begin{picture}(66,22)(0,-32)
\node[Nmarks=i](n0)(16.0,-12.0){}
\node[NLangle=0.0](n1)(36.0,-12.0){$c$}
\node[NLangle=0.0](n2)(56.0,-12.0){$b$}
\node[NLangle=0.0,Nmarks=r](n3)(36.0,-32.0){$ab$}
\node[NLangle=0.0,Nmarks=r](n4)(56.0,-32.0){$a$}
\drawedge(n0,n1){$c+ac$}
\drawedge(n4,n1){$c$}
\drawedge(n3,n1){$c$}
\drawedge(n2,n4){$a$}
\drawedge[curvedepth=1.98](n2,n1){$c$}
\drawedge[curvedepth=1.98](n3,n4){$a$}
\drawedge[curvedepth=1.85](n4,n3){$b$}
\drawedge[curvedepth=1.98](n1,n2){$b$}
\end{picture}}}
\end{tabular}
\end{center}

As given, this automaton accepts one word of length 3 which it should not. It is possible to avoid this, but the resulting automaton is significantly more complicated: it is easier simply to subtract $x^3$ from the generating function it defines. This yields:
\[ s_T(x) = \frac{x^4(2+x)}{1-x-x^2}.\]

\subsubsection{Inflations of \texorpdfstring{$\T$}{T}}
The inflations of the simple permutations depend somewhat on initial and final points in various cells and it is necessary to refine the automaton above to reflect the various cases.  We distinguish 4 types that we can conveniently describe by the first and last code letters of their encodings:
%The inflations are, unfortunately, a bit more subtle than we thought. The easiest way for me to describe it is with reference to the encoding above, and in this situation we end up with four cases:

\begin{itemize}
\item $c\cdots a$ (i.e. code words beginning with $c$ and ending with $a$): In this case, the first point encoded by $a$ and the last point encoded by $c$ inflate to $\Av(213,312)$ and $\Av(132,312)$ respectively.
\item $c\cdots b$: The first point encoded by $a$ inflates to $\Av(213,312)$, and the pair $(b_{last},c_{last})$ is either $(\G\setminus \I,\D)$ or $(\I,\E)$ (symmetry may need to be applied to $\G$ and $\E$).
\item $a\cdots a$: The first two points encoded by $a$ are $(\G\setminus \I,\D)$ or $(\I,\E)$, and the last point of $c$ is $\Av(132,312)$.
\item $a\cdots b$: The first two points encoded by $a$, and the last points encoded by $b$ and $c$ forms pairs which inflate by $(\G\setminus \I,\D)$ or $(\I,\E)$.
\end{itemize}

Let $s_{ca}$, $s_{cb}$ etc denote the generating function for the simples in each of these cases. Modifying the automaton in each case, we get:
\begin{eqnarray*}
s_{ca}(x) = \frac{x^5(2+x)}{1-2x^2-x^3} && s_{cb}=\frac{x^4}{1-2x^2-x^3}\\
s_{aa}(x) = \frac{x^4}{1-2x^2-x^3} && s_{ab}=\frac{x^5}{1-2x^2-x^3}
\end{eqnarray*}

Now let $i_{ca}$, $i_{cb}$ etc denote the generating function of the inflations of the corresponding simple permutations.  We obtain the following expressions. For $i_{ca}$:
\[i_{ca}(x) = s_{ca}\left(\frac{x}{1-x}\right)\cdot \left(\frac{1-x}{x}\right)^2\cdot f_\E(x)^2 = \frac{x^5(2-x)}{(1-2x)^2(1-3x+x^2)(1-x)}\]
Similarly, we obtain: 
\[i_{cb}(x) = = i_{aa} = \frac{x^4(1-2x+x^2)}{(1-2x)^3(1-3x+x^2)}\qquad i_{ab}(x) = \frac{x^5(1-x)^2}{(1-2x)^4(1-3x+x^2)}.\]
Adding them together:
\[ i_\T(x) = \frac{x^4(2-7x+6x^2+x^3-x^4)}{(1-x)(1-2x)^4(1-3x+x^2)} \]

\subsection{The four by seven class \texorpdfstring{$\F$}{F}, and its symmetry class \texorpdfstring{$\S$}{S}}\label{FS-section}

The simples in the class $\F$ lie in one of $2\times 4$ or $4\times 7$.
%given by the following diagram:
%\begin{center}
%\psset{xunit=0.006in, yunit=0.006in} \psset{linewidth=0.005in}
%\begin{pspicture}(0,0)(80,140)
%\drawlines(4,7)
%\cellclass(1,5){$\bullet$}
%\upline(2,7)
%\upline(3,4)
%\upline(4,6)
%\downline(2,3)
%\downline(3,2)
%\downline(4,1)
%\end{pspicture}
%\end{center}
%
To enumerate the $2\times 4$ grid class we use the following encoding and automaton:
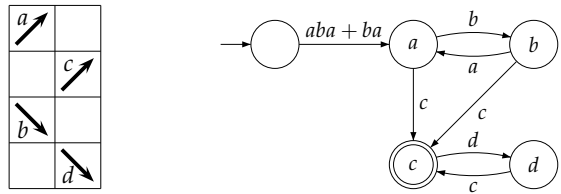
\begin{center}
\begin{tabular}{ccc}
\psset{xunit=0.012in, yunit=0.012in} \psset{linewidth=0.005in}
\begin{pspicture}(0,0)(40,80)
\drawlines(2,4)
\psline[linestyle=solid,linewidth=0.02in]{->}(3,63)(17,77)
\psline[linestyle=solid,linewidth=0.02in]{->}(3,37)(17,23)
\psline[linestyle=solid,linewidth=0.02in]{->}(23,43)(37,57)
\psline[linestyle=solid,linewidth=0.02in]{->}(23,17)(37,3)
\rput(6,74){\footnotesize $a$}
\rput(6,26){\footnotesize $b$}
\rput(26,6){\footnotesize $d$}
\rput(26,54){\footnotesize $c$}
\end{pspicture}
&\rule{10pt}{0pt}&
{\small
\scalebox{0.8}{
\begin{picture}(66,26)(0,-36)
\node[Nmarks=i](n0)(13.0,-12.0){}
\node[NLangle=0.0](n1)(36.0,-12.0){$a$}
\node[NLangle=0.0](n2)(56.0,-12.0){$b$}
\node[NLangle=0.0,Nmarks=r](n3)(36.0,-32.0){$c$}
\node[NLangle=0.0](n4)(56.0,-32.0){$d$}
\drawedge(n0,n1){$aba+ba$}
\drawedge(n1,n3){$c$}
\drawedge(n2,n3){$c$}
\drawedge[curvedepth=1.98](n2,n1){$a$}
\drawedge[curvedepth=1.98](n3,n4){$d$}
\drawedge[curvedepth=1.85](n4,n3){$c$}
\drawedge[curvedepth=1.98](n1,n2){$b$}
\end{picture}}}
\end{tabular}
\end{center}
Notice that this automaton again accepts a spurious length 3 word, and again it is easier simply to delete this from the generating function. After correction, the generating function is
\[ s_{2\times 4}(x) = \frac{x^4(2-x)}{(1-x)^2}.\]

For the $4\times 7$ set we use the following encoding scheme to obtain the generating function $s_{4\times 7}(x)$:
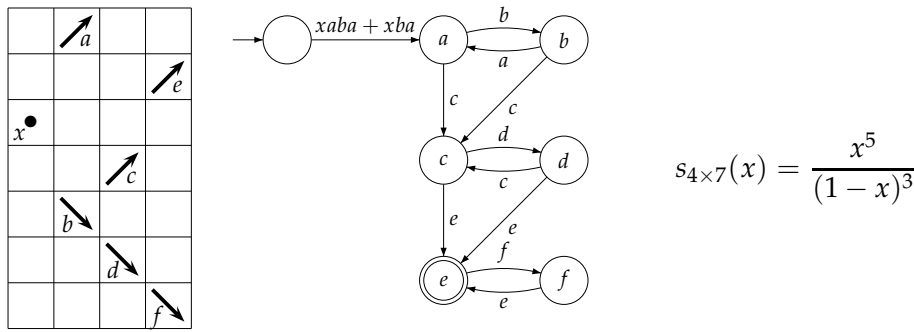
\begin{center}
\begin{tabular}{ccccc}
\psset{xunit=0.012in, yunit=0.012in} \psset{linewidth=0.005in}
\begin{pspicture}(0,0)(80,140)
\drawlines(4,7)
\psline[linestyle=solid,linewidth=0.02in]{->}(23,123)(37,137)
\psline[linestyle=solid,linewidth=0.02in]{->}(23,57)(37,43)
\psline[linestyle=solid,linewidth=0.02in]{->}(43,63)(57,77)
\psline[linestyle=solid,linewidth=0.02in]{->}(43,37)(57,23)
\psline[linestyle=solid,linewidth=0.02in]{->}(63,103)(77,117)
\psline[linestyle=solid,linewidth=0.02in]{->}(63,17)(77,3)
\rput(34,126){\footnotesize $a$}
\rput(26,46){\footnotesize $b$}
\rput(54,66){\footnotesize $c$}
\rput(46,26){\footnotesize $d$}
\rput(74,106){\footnotesize $e$}
\rput(65,5){\footnotesize $f$}
\cellclass(1,5){$\bullet$}
\rput(5,85){\footnotesize $x$}
\end{pspicture}
&\rule{10pt}{0pt}&
{\small
\scalebox{0.8}{
\begin{picture}(64,52)(10,-64)
\node[Nmarks=i](i)(10.0,-16.0){}
\node(a)(36.0,-16.0){$a$}
\node(b)(56.0,-16.0){$b$}
\node(c)(36.0,-36.0){$c$}
\node(d)(56.0,-36.0){$d$}
\node[Nmarks=r](e)(36.0,-56.0){$e$}
\node(f)(56.0,-56.0){$f$}
\drawedge[curvedepth=1.98](a,b){$b$}
\drawedge[curvedepth=1.85](b,a){$a$}
\drawedge[curvedepth=2.12](c,d){$d$}
\drawedge[curvedepth=1.85](d,c){$c$}
\drawedge[curvedepth=1.98](e,f){$f$}
\drawedge[curvedepth=1.98](f,e){$e$}
\drawedge(i,a){$xaba+xba$}
\drawedge(a,c){$c$}
\drawedge(b,c){$c$}
\drawedge(c,e){$e$}
\drawedge(d,e){$e$}
\end{picture}}}
&\rule{10pt}{0pt}&
\psset{xunit=0.012in, yunit=0.012in} 
\begin{pspicture}(0,0)(30,140)
\rput[c](15,70){$\displaystyle s_{4\times 7}(x) = \frac{x^5}{(1-x)^3}$}
\end{pspicture}
\end{tabular}
\end{center}
Combining the two cases, we obtain
\[ s_{\F}(x) = \frac{x^4(2-2x+x^2)}{(1-x)^3}.\]
The enumeration of the simples of class $\S$ is identical by symmetry.

\subsubsection{Inflations of \texorpdfstring{$\F$}{F} and \texorpdfstring{$\S$}{S}}

The analysis of the inflations of the $4\times 7$ subset of $\F$ is quite straightforward.  In any permutation of this set the first point of those encoded by $b$ inflates to $\Av(132, 213)$ while the last point of those encoded by $e$ inflates to $\Av(213, 312)$, and all other points inflate to monotone sequences.  Therefore
\[i_{4\times7}(x) = s_{4\times 7}\left(\frac{x}{1-x}\right) \cdot \left(\frac{1-x}{x}\right)^2\cdot f_\E(x)^2 = \frac{x^5}{(1-2x)^5}\]

However, the analysis of the $2\times 4$ subset is more subtle and, just as for the inflations of the permutations of $\T$, it is necessary to distinguish several types, refine the automaton for each type and determine how the points of each type are inflated.  Nevertheless the analysis is similar to that of class $\T$ and so we merely summarize the derivation of the generating functions of the inflations of each type.  We divide into 7 types.  In the following table we give the code words for each type, the corresponding generating function for the simples, the factor of the generating function for the inflations arising from points that do not inflate monotonically, and the resulting contribution to this generating function.

\centerline{
\begin{tabular}{|llll|}
\hline
Code words&Simples&Non-monotone factor&Inflations\\
\hline
$aba\cdots cdc$&$\frac{x^6}{(1-x)^2(1+x)}$&$f_\E^2$&$\frac{x^6}{(1-2x)^4(1-x)}$\\
$aba\cdots bc$&$\frac{x^5}{(1-x^2)}$&$f_\E^2$&$\frac{x^5}{(1-2x)^3(1-x)}$\\
$aba\cdots ac$&$\frac{x^4}{(1-x^2)}$&$f_\E(f_\G+f_\E- x/(1-x))$&$\frac{x^4(1-x)}{(1-2x)^4}$\\
$bac\cdots dc$&$\frac{x^5}{(1-x^2)}$&$f_\E f_\G$&$\frac{x^5(1-3x+3x^2)}{(1-2x)^4(1-x)^2}$\\
$ba \cdots bc$&$\frac{x^4}{(1-x^2)}$&$f_\E (f_\G + f_\E - x/(1-x))$&$\frac{x^4(1-x)}{(1-2x)^4}$\\
$bab \cdots ac$&$\frac{x^5}{(1-x^2)}$&$(f_\G + f_\E- x/(1-x))^2$&$\frac{x^5(1-x)}{(1-2x)^5}$\\
$bab \cdots cdc$&$\frac{x^6}{(1-x)^2(1+x)}$&$f_\E (f_\G + f_\E - x/(1-x))$&$\frac{x^6}{(1-2x)^5}$\\
\hline
\end{tabular}
}

%We preserve the division of $F$ into the $4\times 7$ class, and the $2\times 4$ degenerate one. The situation is as subtle as the class $T$, so we again will rely on distinguishing cases according to the encodings.
%
%\paragraph{The $2\times 4$ grid class:}For the $2\times 4$ class, we identify 7 distinct cases. In each case, I'm merely giving the non-monotone factor needed, and the resulting generating function of the inflation obtained.
%\begin{enumerate}
%\item $aba\cdots cdc$: $f_E^2$. Simples: $\frac{x^6}{(1-x)^2(1+x)}$. Inflation: $\frac{x^6}{(1-2x)^4(1-x)}$
%\item $aba\cdots bc$: $f_E^2$. Simples: $\frac{x^5}{(1-x^2)}$. Inflation: $\frac{x^5}{(1-2x)^3(1-x)}$
%\item $aba\cdots ac$: $f_E(f_G+f_E- x/(1-x))$. Simples: $\frac{x^4}{(1-x^2)}$. Inflation: $\frac{x^4(1-x)}{(1-2x)^4}$
%\item $bac\cdots dc$: $f_Ef_G$. Simples: $\frac{x^5}{(1-x^2)}$. Inflation: $\frac{x^5(1-3x+3x^2)}{(1-2x)^4(1-x)^2}$
%\item $ba \cdots bc$: $f_E(f_G+f_E- x/(1-x))$. Simples: $\frac{x^4}{(1-x^2)}$. Inflation: $\frac{x^4(1-x)}{(1-2x)^4}$
%\item $bab \cdots ac$: $(f_G+f_E- x/(1-x))^2$. Simples: $\frac{x^5}{(1-x^2)}$. Inflation: $\frac{x^5(1-x)}{(1-2x)^5}$
%\item $bab \cdots cdc$: $f_E(f_G+f_E- x/(1-x))$. Simples: $\frac{x^6}{(1-x)^2(1+x)}$. Inflation: $\frac{x^6}{(1-2x)^5}$
%\end{enumerate}
Summing the contributions, we get:
\[i_{2\times 4}(x) = \frac{x^4(2-7x+7x^2+x^3-4x^4)}{(1-x)^2(1-2x)^5}.\]

Consequently
\[i_\F (x)=i_{4\times 7}(x)+i_{2\times 4}(x)=\frac{x^4(2-6x+5x^2+2x^3-4x^4)}{(1-x)^2(1-2x)^5}.\]
%\paragraph{The $4\times 7$ class:} Miraculously, no complications arise here. The first point of $b$ is inflated by $\Av(132,312)$, and the last point of $e$ is inflated by $\Av(213,312)$. All other points inflate monotonically. Thus:
%\[i_{4\times7}(x) = s_{4\times 7}\left(\frac{x}{1-x}\right) \cdot \left(\frac{1-x}{x}\right)^2\cdot f_E(x)^2 = \frac{x^5}{(1-2x)^5}\]
%And so the inflations of the simples in $F$ are given by:
%\[ i_F(x) = \frac{x^4(2-6x+5x^2+2x^3-4x^4)}{(1-x)^2(1-2x)^5}.\]

\subsection{The six by six class \texorpdfstring{$\X$}{X}}

In general we shall encode the permutations of $\X$ by the following scheme:

%This is rather more complicated. The general encoding scheme we will use is
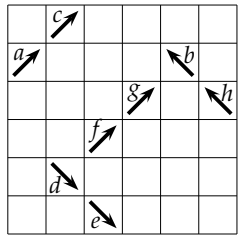
\begin{center}
\psset{xunit=0.01in, yunit=0.01in} \psset{linewidth=0.005in}
\begin{pspicture}(0,0)(120,120)
\drawlines(6,6)
\psline[linestyle=solid,linewidth=0.02in]{->}(3,83)(17,97)
\psline[linestyle=solid,linewidth=0.02in]{->}(23,103)(37,117)
\psline[linestyle=solid,linewidth=0.02in]{->}(23,37)(37,23)
\psline[linestyle=solid,linewidth=0.02in]{->}(43,43)(57,57)
\psline[linestyle=solid,linewidth=0.02in]{->}(43,17)(57,3)
\psline[linestyle=solid,linewidth=0.02in]{->}(63,63)(77,77)
\psline[linestyle=solid,linewidth=0.02in]{->}(97,83)(83,97)
\psline[linestyle=solid,linewidth=0.02in]{->}(117,63)(103,77)
\rput(6,94){\footnotesize $a$}
\rput(95,94){\footnotesize $b$}
\rput(26,114){\footnotesize $c$}
\rput(25,26){\footnotesize $d$}
\rput(46,6){\footnotesize $e$}
\rput(46,54){\footnotesize $f$}
\rput(66,74){\footnotesize $g$}
\rput(115,74){\footnotesize $h$}
\end{pspicture}
\end{center}

However there are some degenerate cases when certain cells are empty that it is convenient to handle separately.
%When certain cells contain no points, we obtain a degenerate class.
For example, if there are no points encoded by $a$ or by $b$, then the simple permutations are the same as given by the $2\times 4$ set above. Similarly, if there are no points encoded by $c$ or by $d$, then we have the reverse-inverse-reverse of that set, which has the same generating function.

Thus, we will consider the case that there are points encoded by all of $a,b,c$ and $d$. If there are no other code letters, then we are enumerating points in the $3\times 3$ class. We note that the class $3\times 3$ contains the wedge simple permutations of both varieties, but these have already been counted by $s_{2\times 4}$ or its symmetry so we do not wish to count these again. In the automaton below, this is achieved by requiring that accepted words must contain at least two $b$s and at least two $d$s. We thus obtain:
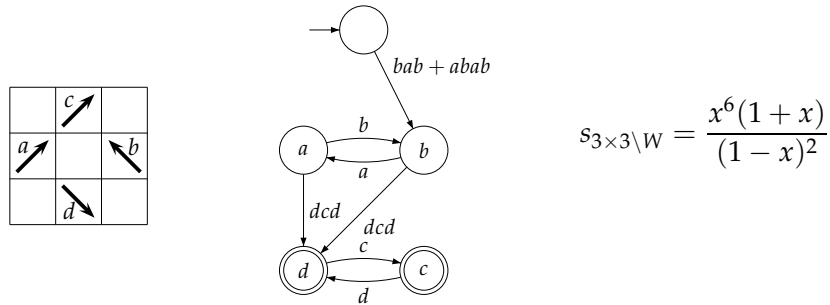
\begin{center}
\begin{tabular}{ccccc}
\psset{xunit=0.012in, yunit=0.012in} \psset{linewidth=0.005in}
\begin{pspicture}(0,-20)(60,60)
\drawlines(3,3)
\psline[linestyle=solid,linewidth=0.02in]{->}(3,23)(17,37)
\psline[linestyle=solid,linewidth=0.02in]{->}(23,43)(37,57)
\psline[linestyle=solid,linewidth=0.02in]{->}(23,17)(37,3)
\psline[linestyle=solid,linewidth=0.02in]{->}(57,23)(43,37)
\rput(6,34){\footnotesize $a$}
\rput(26,54){\footnotesize $c$}
\rput(26,6){\footnotesize $d$}
\rput(54,34){\footnotesize $b$}
\end{pspicture}
&\rule{10pt}{0pt}&
{\small
\scalebox{0.8}{
\begin{picture}(56,40)(26,-32)
\node[Nmarks=i](n0)(46.0,8.0){}
\node[NLangle=0.0](n1)(36.0,-12.0){$a$}
\node[NLangle=0.0](n2)(56.0,-12.0){$b$}
\node[NLangle=0.0,Nmarks=r](n3)(36.0,-32.0){$d$}
\node[NLangle=0.0,Nmarks=r](n4)(56.0,-32.0){$c$}
\drawedge(n0,n2){$bab+abab$}
\drawedge(n1,n3){$dcd$}
\drawedge(n2,n3){$dcd$}
\drawedge[curvedepth=1.98](n2,n1){$a$}
\drawedge[curvedepth=1.98](n3,n4){$c$}
\drawedge[curvedepth=1.85](n4,n3){$d$}
\drawedge[curvedepth=1.98](n1,n2){$b$}
\end{picture}}}
&\rule{10pt}{0pt}&
\psset{xunit=0.012in, yunit=0.012in} \psset{linewidth=0.005in}
\begin{pspicture}(0,-20)(30,60)
\rput[c](15,40){$\displaystyle s_{3\times 3\setminus W}=\frac{x^6(1+x)}{(1-x)^2}$}
\end{pspicture}
\end{tabular}
\end{center}

%N.B. this does \emph{not} enumerate all simples in the class, as we have deliberately excluded the wedge simple permutations.

%Happily, it seems to me that the rule is now a bit simpler: "at least
%two bs, and at least two ds". There are some other rules, too, such as
%"points which could be encoded either a or c should be encoded by a."

We now turn to the enumeration of the remaining simples in the class, namely those involving all of the code letters $a$, $b$, $c$ and $d$, and  at least one other letter from $e$, $f$, $g$ and $h$ (these simples all belong to the set called $6\times 6$). The following automaton is actually less complex than it looks at first sight.  Apart from ensuring that all $a$'s and $b$'s come before all $c$'s and $d$'s etc., it also enforces that the same permutation is not coded by different words.  For example, the first $d$ comes before the first $c$ otherwise the first $c$ could have been coded by an $a$ with appropriate redrawing of gridlines.  This and similar restrictions gives us the automaton and generating function:

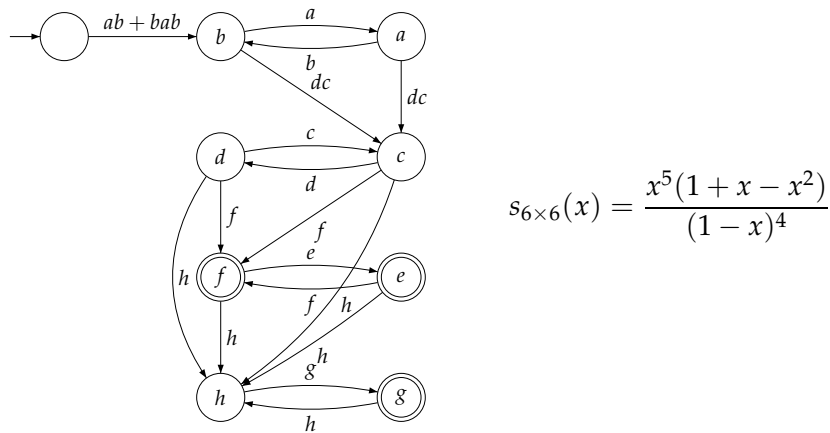
\begin{center}
\begin{tabular}{ccc}
{\small
\scalebox{0.8}{
\begin{picture}(64,60)(0,-76)
\node[Nmarks=i](i)(10.0,-16.0){}
\node(a)(66.0,-16.0){$a$}
\node(b)(36.0,-16.0){$b$}
\node(c)(66.0,-36.0){$c$}
\node(d)(36.0,-36.0){$d$}
\node[Nmarks=r](e)(66.0,-56.0){$e$}
\node[Nmarks=r](f)(36.0,-56.0){$f$}
\node[Nmarks=r](g)(66.0,-76.0){$g$}
\node(h)(36.0,-76.0){$h$}
\drawedge[curvedepth=1.98](a,b){$b$}
\drawedge[curvedepth=1.85](b,a){$a$}
\drawedge[curvedepth=2.12](c,d){$d$}
\drawedge[curvedepth=1.85](d,c){$c$}
\drawedge[curvedepth=1.98](e,f){$f$}
\drawedge[curvedepth=1.98](f,e){$e$}
\drawedge[curvedepth=1.98](g,h){$h$}
\drawedge[curvedepth=1.98](h,g){$g$}
\drawedge(i,b){$ab+bab$}
\drawedge(a,c){$dc$}
\drawedge(b,c){$dc$}
\drawedge(c,f){$f$}
\drawedge(d,f){$f$}
\drawedge[curvedepth=5](c,h){$h$}
\drawedge[curvedepth=-8](d,h){$h$}
\drawedge[curvedepth=1](e,h){$h$}
\drawedge(f,h){$h$}
\end{picture}}}
&\rule{10pt}{0pt}&
\psset{xunit=0.025in, yunit=0.025in} \psset{linewidth=0.005in}
\begin{pspicture}(0,0)(80,80)
\rput[c](40,40){$\displaystyle s_{6\times 6}(x) = \frac{x^5(1+x-x^2)}{(1-x)^4}$}
\end{pspicture}
\end{tabular}
\end{center}

We now need to combine the three generating functions given above, taking into consideration any intersections to avoid overcounting. The simples enumerated by $s_{3\times3\setminus W}(x)$ deliberately exclude the  wedge simple permutations and, because of this, their intersection with the simples enumerated by $s_{2\times4}(x)$ and its symmetry is empty. Similarly, the definition of the class $6\times 6$ was such that it was disjoint from the other subcases. Consequently, the only intersection that we need to consider is $2\times 4$ with its symmetry under reverse-inverse-reverse. It is straightforward to check that the only simples in this intersection are 2413, 3142 and 24153. The generating function for the simples in $X$ is therefore:
\[s_X(s) = s_{6\times 6}(x)+2s_{2\times 4}(x)+s_{3\times 3\setminus W}(x) - 2x^4 - x^5 = \frac{x^4(2-2x+2x^2-2x^3+x^4)}{(1-x)^4}.\]

\subsubsection{Inflations of \texorpdfstring{$\X$}{X}}
First, we shall take the inflations of $6\times6$ into account: all but two points are inflated monotonically, and the remaining two (last point of $b$, and first point of $d$) are inflated by (symmetries of) $\E$. Hence
\[ i_{6\times 6}(x) = s_{6\times6}\left(\frac{x}{1-x}\right)\cdot\left(\frac{1-x}{x}\right)^2\cdot f_\E (x)^2 = \frac{x^5(1-x-x^2)}{(1-x)(1-2x)^6}.\]

Next, Section \ref{FS-section} provides the inflations of $2\times 4$ and its symmetry. This leaves the $3\times3$ class, which splits into four cases summarized below in the same way we summarized the $2\times 4$ class:

\centerline{
\begin{tabular}{|llll|}
\hline
Code words&Simples&Non-monotone factor&Inflations\\
\hline
$bab\cdots cd$&$\frac{x^6}{(1-x)^2(1+x)}$&$f_\E^2$&$\frac{x^6}{(1-2x)^4(1-x)}$\\
$bab\cdots dc$&$\frac{x^7}{(1-x)^2(1+x)}$&$f_\E(f_\E+f_\G- x/(1-x))$&$\frac{x^7}{(1-2x)^5(1-x)}$\\
$aba\cdots cd$&$\frac{x^7}{(1-x)^2(1+x)}$&$f_\E(f_\E+f_\G- x/(1-x))$&$\frac{x^7}{(1-2x)^5(1-x)}$\\
$aba\cdots dc$&$\frac{x^8}{(1-x)^2(1+x)}$&$(f_\E+f_\G-x/(1-x))^2$&$\frac{x^8}{(1-2x)^6(1-x)}$\\
\hline
\end{tabular}
}

%\begin{enumerate}
%\item $bab\cdots cd$: $f_E^2$. Simples: $\frac{x^6}{(1-x)^2(1+x)}$. Inflation: $\frac{x^6}{(1-2x)^4(1-x)}$.
%\item $bab\cdots dc$: $f_E(f_E+f_G-x/(1-x))$. Simples: $\frac{x^7}{(1-x)^2(1+x)}$. Inflation: $\frac{x^7}{(1-2x)^5(1-x)}$.
%\item $aba\cdots cd$: $f_E(f_E+f_G-x/(1-x))$. Simples: $\frac{x^7}{(1-x)^2(1+x)}$. Inflation: $\frac{x^7}{(1-2x)^5(1-x)}$.
%\item $aba\cdots dc$: $(f_E+f_G-x/(1-x))^2$. Simples: $\frac{x^8}{(1-x)^2(1+x)}$. Inflation: $\frac{x^8}{(1-2x)^6(1-x)}$.
%\end{enumerate}
Adding these contributions together yields:
\[i_{3\times3\setminus W}(x) = \frac{x^6(1-x)}{(1-2x)^6}\]
Recall again that this excludes inflations of wedge simple permutations (which we counted in the $2\times 4$ class).

We also need the generating function for the inflations of $2413$, $3142$ and $24153$, and this is: \begin{multline*}2\left(\frac{x}{1-x}\right)^2\cdot f_\E(x)\cdot\left(f_\E(x)+f_\G(x)-\frac{x}{1-x}\right) + \\\left(\frac{x}{1-x}\right)^3\cdot\left(f_\E(x)+f_\G(x)-\frac{x}{1-x}\right)^2 =\frac{x^4(2-3x)}{(1-x)(1-2x)^4}\end{multline*}

Putting everything together:
\begin{eqnarray*}i_\X(x) &=& i_{6\times 6}(x) + 2i_{2\times 4}(x)+i_{3\times 3\setminus W}(x) - \frac{x^4(2-3x)}{(1-x)(1-2x)^4}\\
&=& \frac{x^4(x^7-26x^6+46x^5-14x^4-25x^3+28x^2-12x+2)}{(1-x)^4(1-2x)^6}
\end{eqnarray*}

\subsection{The strong-indecomposables}

The strong-indecomposables are the inflations of simple permutations of length 4 or more.  Having calculated the inflations in each of $\T$, $\F$, $\S$ and $\X$ the main ingredient we still need is to compute the inflations of  intersections between these classes.

\paragraph{Inflations of wedge simples}
In both varieties of wedge simples there are two of each length from 4 onwards but the two varieties intersect in $\{2413, 3142, 24153\}$.  In computing their inflations we have to distinguish between odd and even lengths.  In even lengths both the permutations of each variety have one point that inflates to $\E$ (or a symmetry) and a pair of points that inflates to $(\G\setminus \I,\D)$ or $(\I,\E)$.  But in odd lengths each variety has one wedge simple with two points inflating by (symmetries of) $\E$, whereas the other has two pairs of points which inflate to $(\G\setminus \I,\D)$ or $(\I,\E)$.  Taking all this into consideration a routine calculation shows that

\[i_W(x)=\frac{x^4(2-5x+2x^2+2x^3)}{(1-x)(1-2x)^5}\]

%There are two wedge simples of the form $<\hspace{-5pt}\cdot$ of each length from 4 onwards. The parity of the length matters: for an even number of points, we find both permutations give us one point inflating to (a symmetry of) $E$, and a pair of points which inflates to $(G\setminus I,D)$ or $(I,E)$. For an odd number of points, one wedge simple has two points inflating by (symmetries of) $E$, and the other has two pairs of points which inflate to $(G\setminus I,D)$ or $(I,E)$.
%\[i_W(x) = \frac{2x^4(1-x)}{(1-2x)^4} + \frac{x^5(2-6x+5x^2)}{(1-2x)^5(1-x)} = \frac{x^4(2-6x+4x^2+x^3)}{(1-2x)^5(1-x)}\]

We now consider the intersections:
\begin{itemize}
\item $\T\cap \X$: By inspection, these are the permutations in the $3\times 3$ class given above.
%We only have $i_{3\times 3\setminus W}(x)$, so we need to add the wedge simples back in. We need to be slightly careful here, however. The $3\times3$ class omits wedge simples of two varieties, so we need two copies of $i_W(x)$, but then we overcount the two length $4$ permutations and $24153$ (and their inflations).
We have, \[i_{3\times 3}(x) = i_{3\times3\setminus W}(x) + i_W(x) = \frac{(2-9x+13x^2-4x^3-3x^4)x^4}{(1-2x)^6(1-x)}\]

\item $\F\cap \X$ and $\S\cap \X$.
The permutation in $\F\cap \X$ lie in the class
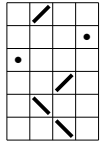
\begin{center}
\psset{xunit=0.006in, yunit=0.006in} \psset{linewidth=0.005in}
\begin{pspicture}(0,0)(80,120)\drawlines(4,6)\upline(2,6)\upline(3,3)\downline(2,2)\downline(3,1)\cellclass(1,4){\tiny$\bullet$}%
\cellclass(4,5){\tiny$\bullet$}
\end{pspicture}\end{center}
Following similar analysis to the class $\F$, the simples lie either in the $2\times 4$ class, or all cells in the above cell diagram are nonempty. We have already considered $ i_{2\times 4}(x)$, while for the second type (which we call the $4\times 6$ class), we observe that the automaton to enumerate its simples  is the same as for the  $4\times 7$ class, but with no points encoded by the letter $f$. Thus $s_{4\times 6}(x) = s_{4\times 7}(x)\cdot (1-x^2) = x^5(1+x)/(1-x)^2$.

To compute $i_{4\times 6}(x)$ we note that   points in the  $4\times 6$ class are inflated in the same way as for the $4\times 7$ class.  This gives
%As before, the permutations in $F\cap X$ lie in a $4\times 6$ class. We had $s_{F\cap X}(x) = s_{4\times6}(x)+s_{2\times 4}(x)$, and we inflate using this division too. We already have $i_{2\times 4}(x)$, and we inflate points in the full $4\times 6$ class in the same way as we inflated for the $4\times 7$ class:
\[ i_{4\times 6}(x) = s_{4\times 6}\left(\frac{x}{1-x}\right) \cdot \left(\frac{1-x}{x}\right)^2\cdot f_\E(x)^2 = \frac{x^5}{(1-x)^2(1-2x)^4}.\]
%Combining this with $i_{2\times 4}(x)$ we get:
Therefore
\[i_{\F\cap \X}(x) =  i_{4\times 6}(x)+ i_{2\times 4}(x)=\frac{x^4(2-6x+5x^2+x^3-4x^4)}{(1-2x)^5(1-x)^2}.\]
$\S\cap \X$ is handled by symmetry.
\item $\T \cap \F$ consists of wedge simples shaped like $<\hspace{-5pt}\cdot$ while $\T\cap \S$ is the other variety of wedge simples.  Both sets are obviously contained in $\X$ and so $\T\cap \F=\T\cap \F\cap \X$ and $\T\cap \S=\T\cap \S\cap \X$.
\item $\F\cap \S$ and $\T\cap \F\cap \S$ (which are both finite) are also obviously contained in $\X$ and so $\F\cap \S=\F\cap \S\cap \X$ and $\T\cap \F\cap \S=\T\cap \F\cap \S\cap \X$:
%this class contains $2413, 3142, 24153$ which we considered above and, in addition, 31524 and 351624.
%this class contains two simple permutations of each length 4 and 5, and one of length 6, which we need to inflate. We've done the length 4 ones before, and the length 5 permutations are 24153 (which we've done) and 31524, and the one of length 6 is 351624.
%In each of these last two, all but two points inflate  monotonically) and the last two inflate by symmetries of $E$. A standard calculation gives:
%%    \[ i_{F\cap S} (x) = \frac{x^4(2-3x)}{(1-x)(1-2x)^4} + \left(1+\frac{x}{1-x}\right)\left(\frac{x}{1-x}\right)^3\cdot f_E(x)^2 = \frac{x^4(3-11x+12x^2-3x^3)}{(1-x)^3(1-2x)^4}\]
%    \[ i_{F\cap S} (x) = \frac{x^4(3-11x+12x^2-3x^3)}{(1-x)^3(1-2x)^4}\]
%
%\item $T\cap F\cap S$: Two of length 4, and one of length 5 (24153). We've done this before, it was: \[\frac{x^4(2-3x)}{(1-x)(1-2x)^4}\]
%
%\item $T\cap F\cap X$ and $T\cap S\cap X$: Wedge simples again, $i_W(x)$ (so we didn't need to work this out).
%\item $F\cap S\cap X$: same as $F\cap S$, $\frac{x^4(3-11x+12x^2-3x^3)}{(1-x)^3(1-2x)^4}$. (So we didn't need to work out these out either)
%\item $T\cap F\cap S \cap X$: $\frac{x^4(2-3x)}{(1-x)(1-2x)^4}$, cancelling with $T\cap F\cap S$.
\end{itemize}

These set equalities simplify the inclusion-exclusion calculation of the strong-indecomposables to
\begin{align*}
\begin{split}
i(x)&=i_\T+i_\F+i_S+i_\X\\
&\quad -(i_{\T\cap \F}+i_{\T\cap \S}+i_{\T\cap \X}+i_{\F\cap \S}+i_{\F\cap \X}+i_{\S\cap \X})\\
&\quad +(i_{\T\cap \F\cap \S}+i_{\T\cap \F\cap \X}+i_{\T\cap \S\cap \X}+i_{\F\cap \S\cap \X})\\
&\quad -i_{\T\cap \F\cap \S\cap \X}\\
&= i_\T+i_\F+i_\S+i_\X
- \left(  i_{3\times 3} + i_{\F\cap \X} + i_{\S\cap \X}\right)\\
\end{split}
\end{align*}
Here and henceforth the omission of a subscript referring to a set on a generating function means that we are dealing with the entire class $\Av(1324,4312)$.
%\[
%i_{\Av(1324,4312)}(x) = i_T+i_F+i_S+i_X
%- \left(  i_{3\times 3} + i_{F\cap X} + i_{S\cap X}\right)
%\]
This yields:
%\[\frac{(16x^6-57x^5+41x^4+8x^3-24x^2+12x-2)x^4}{(x-1)(1-3x+x^2)(1-2x)^6}\]
\[
i(x) = \frac{(2-12x+24x^2-8x^3-41x^4+57x^5-16x^6)x^4}{(1-x)(1-3x+x^2)(1-2x)^6}
\]
%The sequence (from $n=4$) goes 2, 20, 120, 570, 2355, 8841, 30906, 102187.

\subsection{Sum- and skew-decomposables, and final enumeration}

We begin with the sum-decomposables. These are of the form $\alpha_1\oplus \alpha_2 \oplus \cdots \oplus \alpha_k$, where $\alpha_2,\dots,\alpha_{k-1}$ are singletons, $\alpha_1\in \Av(132,4312)$, $\alpha_k\in\Av(213,4312)=\G$ and $\alpha_1,\alpha_k$ are sum-indecomposable. These last two two classes are symmetric, so we need enumerate only the sum-indecomposables of $\G$. %To make the expression unique, we need the sum-indecomposables of the class $G$.

Because $\G$ avoids $213$, sum-decomposable elements of $\G$ are of the form $1\oplus \cdots \oplus 1 \oplus \pi$, where $\pi$ is sum-indecomposable. So we have the equation, $f_\G(x) = f_{\not\oplus_\G(x)} \cdot \frac{1}{1-x}$, and hence
\[
f_{\not\oplus_\G(x)} = \frac{x(1-3x+3x^2)}{(1-2x)^2}.
\]
This gives
\[
f_{\oplus}(x) = f_{\not\oplus_\G(x)}(x)^2 \cdot \frac{1}{1-x} = \frac{x^2(1-3x+3x^2)^2}{(1-2x)^4(1-x)}
\]

Now we analyze the skew-decomposables. These take the form $\alpha_1\ominus \alpha_2 $, where $\alpha_1$ is $\ominus$-indecomposable, and either $\alpha_1\in \I$ and $\alpha_2\in \Av(312,1324)$, or $\alpha_1\in \Av(1324,4312)\setminus \I$ and $\alpha_2\in \D$.
%(N.B. I have modified Mike's description a bit, but I think what I have written is equivalent.)

We already know that  $f_{\Av(312,1324)}=\frac{x(1-3x+3x^2)}{(1-x)(1-2x)^2}$. So now we need the expression for the skew-indecomposables in $\Av(1324,4312)$, which essentially consists of everything we have counted so far:
\begin{align*}
f_{\not\ominus}(x) &= x + {f_\oplus}(x) + i(x)\\
 &= \frac{x(1-15x+99x^2-373x^3+879x^4-1338x^5+1311x^6-804x^7+289x^8-44x^9)}{(1-3x+x^2)(1-2x)^6(1-x)}.\end{align*}
To exclude the class of increasing permutations from this enumeration we just have to subtract $x/(1-x)$.

The skew-decomposables are therefore enumerated by
\[
f_{\ominus}(x) = \frac{x^2(1-14x+87x^2-313x^3+718x^4-1090x^5+1095x^6-708x^7+273x^8-44x^9)}{(1-3x+x^2)(1-2x)^6(1-x)^2}.
\]

Finally, the class $\Av(1324, 4312)$ itself is enumerated  by
%\begin{align*} f_{\Av(1324,4312)}(x) &= \not\hspace{-3pt}\ominus_{\Av(1324,4312)}(x) + \ominus_{\Av(1324,4312)}(x)\\[5pt]
%&= \frac{-x(60x^9-385x^8+1020x^7-1559x^6+1499x^5-939x^4+385x^3-100x^2+15x-1)}{(x^2-3x+1)(1-2x)^6(1-x)^2}.\end{align*}
%
\begin{align*}
f(x) &= f_{\not\ominus}(x) + f_{\ominus}(x)\\[5pt]
&= \frac{x(1-15x+100x^2-385x^3+939x^4-1499x^5+1559x^6-1020x^7+385x^8-60x^9)}{(1-3x+x^2)(1-2x)^6(1-x)^2}.\end{align*}

The series expansion begins 1, 2, 6, 22, 86, 335, 1266, 4598, 16016, 53579, 172663, 537957, $\dots$ (sequence A165526 of~\cite{sloane:the-on-line-enc:}).

\bibliographystyle{acm}
\bibliography{refs}

\end{document}